\documentclass{amsart}
\usepackage[all]{xy}
\usepackage{verbatim}
\usepackage{color}
\usepackage{amsthm}
\usepackage{amssymb}
\usepackage[colorlinks=true]{hyperref}

\numberwithin{equation}{section}

\newtheorem{theorem}[equation]{Theorem}
\newtheorem*{theorem*}{Theorem} \newtheorem{lemma}[equation]{Lemma}

\newtheorem*{conjecture*}{Mamma Conjecture}
\newtheorem*{conjecture1*}{Mamma Conjecture (revisited)}
\newtheorem{proposition}[equation]{Proposition}
\newtheorem{corollary}[equation]{Corollary}
\newtheorem*{corollary*}{Corollary}

\theoremstyle{remark}
\newtheorem{definition}[equation]{Definition}

\newtheorem{example}[equation]{Example}

\theoremstyle{remark}
\newtheorem{remark}[equation]{Remark}

\newcommand{\cA}{{\mathcal A}}
\newcommand{\cB}{{\mathcal B}}
\newcommand{\cC}{{\mathcal C}}
\newcommand{\cD}{{\mathcal D}}

\newcommand{\cF}{{\mathcal F}}
\newcommand{\cG}{{\mathcal G}}
\newcommand{\cH}{{\mathcal H}}

\newcommand{\cM}{{\mathcal M}}

\newcommand{\cO}{{\mathcal O}}
\newcommand{\cP}{{\mathcal P}}

\newcommand{\bbA}{\mathbb{A}}
\newcommand{\bbB}{\mathbb{B}}
\newcommand{\bbC}{\mathbb{C}}

\newcommand{\bbG}{\mathbb{G}}

\newcommand{\bbP}{\mathbb{P}}
\newcommand{\bbR}{\mathbb{R}}

\newcommand{\bbT}{\mathbb{T}}
\newcommand{\bbQ}{\mathbb{Q}}
\newcommand{\bbZ}{\mathbb{Z}}

\DeclareMathOperator{\Mod}{Mod}
\DeclareMathOperator{\Mot}{Mot}

\newcommand{\dgcat}{\mathrm{dgcat}} 

\newcommand{\perf}{\mathrm{perf}}

\newcommand{\dg}{\mathrm{dg}}

\newcommand{\uHom}{\underline{\mathrm{Hom}}}
\newcommand{\Hom}{\mathrm{Hom}}

\newcommand{\dgHo}{\mathrm{H}^0}

\newcommand{\Ho}{\mathrm{Ho}}

\newcommand{\Hmo}{\mathrm{Hmo}}
\newcommand{\op}{\mathrm{op}}

\newcommand{\too}{\longrightarrow}

\newcommand{\ie}{\textsl{i.e.}\ }
\newcommand{\eg}{\textsl{e.g.}}

\let\oldmarginpar\marginpar
\def\marginpar#1{\oldmarginpar{\tiny #1}}

\begin{document}

\title[Modified mixed realizations, additive invariants, and periods]{Modified mixed realizations, new additive invariants, and periods of dg categories}
\author{Gon{\c c}alo~Tabuada}

\address{Gon{\c c}alo Tabuada, Department of Mathematics, MIT, Cambridge, MA 02139, USA}
\email{tabuada@math.mit.edu}
\urladdr{http://math.mit.edu/~tabuada}
\thanks{The author was partially supported by a NSF CAREER Award}

\date{\today}


\abstract{To every scheme, not necessarily smooth neither proper, we can associate its different mixed realizations (de Rham, Betti, \'etale, Hodge, etc) as well as its ring of periods. In this note, following an insight of Kontsevich, we prove that, after suitable modifications, these classical constructions can be extended from schemes to the broad setting of dg categories. This leads to new additive invariants of dg categories, which we compute in the case of differential operators, as well as to a theory of periods of dg categories.  Among other applications, we prove that the ring of periods of a scheme is invariant under projective homological duality. Along the way, we explicitly describe the modified mixed realizations using the Tannakian formalism.}
}

\maketitle
\vskip-\baselineskip
\vskip-\baselineskip



%
%
%
%
\section{Modified mixed realizations}
Given a perfect field $k$ and a commutative $\bbQ$-algebra~$R$, Voevodsky introduced in \cite[\S2]{Voevodsky} the category of geometric mixed motives $\mathrm{DM}_{\mathrm{gm}}(k;R)$. By construction, this $R$-linear rigid symmetric monoidal triangulated category comes equipped with a $\otimes$-functor $M(-)_R\colon \mathrm{Sm}(k) \to \mathrm{DM}_{\mathrm{gm}}(k;R)$, defined on smooth $k$-schemes of finite type, and with a $\otimes$-invertible object $\bbT:=R(1)[2]$ called the {\em Tate motive}. Moreover, when $k$ is of characteristic zero, the preceding functor can be extended from $\mathrm{Sm}(k)$ to the category $\mathrm{Sch}(k)$ of all $k$-schemes of finite type. Recall also the construction of Voevodsky's big category of mixed motives $\mathrm{DM}(k;R)$. This $R$-linear symmetric monoidal triangulated category admits arbitrary direct sums and $\mathrm{DM}_{\mathrm{gm}}(k;R)$ identifies with its full triangulated subcategory of compact objects.

A {\em differential graded (=dg) category $\cA$}, over a base field $k$, is a category enriched over complexes of $k$-vector spaces; see \S\ref{sub:dg}. Every (dg) $k$-algebra $A$ gives naturally rise to a dg category with a single object. Another source of examples is provided by schemes since the category of perfect complexes of every quasi-compact quasi-separated $k$-scheme $X$ admits a canonical dg enhancement $\perf_\dg(X)$; see \cite[\S4.4]{ICM-Keller}. Let us denote by $\dgcat(k)$ the category of (small) dg categories and by $\Hmo(k)$ its localization at the class of derived Morita equivalences.

Given an $R$-linear symmetric monoidal additive category with arbitrary direct sums $\cM$ and a $\otimes$-invertible object $\cO \in \cM$, consider the commutative monoid $\oplus_{m \in \bbZ} \cO^{\otimes m}$ in $\cM$ and the category of (right) $\oplus_m \cO^{\otimes m}$-modules $\mathrm{Mod}(\oplus_m \cO^{\otimes m})$. In what follows, we write $\gamma\colon \cM \to \mathrm{Mod}(\oplus_m \cO^{\otimes m})$ for the base-change functor. 
\begin{definition}[(Modified) mixed realization]\label{def:modified}
A {\em mixed realization} is an $R$-linear lax $\otimes$-functor $H\colon \mathrm{DM}(k;R) \to \cM$ such that $H(\oplus_m {\bbT}^{\otimes m})\simeq \oplus_m H(\bbT)^{\otimes m}$. The associated {\em modified mixed realization} is the following composition
$$ \mathrm{H}\colon \mathrm{Sm}(k) \stackrel{M(-)_R}{\too} \mathrm{DM}_{\mathrm{gm}}(k;R) \stackrel{(-)^\vee}{\too} \mathrm{DM}_{\mathrm{gm}}(k;R) \stackrel{H}{\too} \cM \stackrel{\gamma}{\too} \mathrm{Mod}(\oplus_mH(\bbT)^{\otimes m})\,,$$
where $(-)^\vee$ stands for the (contravariant) duality autoequivalence. 
\end{definition}
In what follows, given a smooth $k$-scheme of finite type $X$, we will write $H(X)$ instead of $H(M(X)_R^\vee)$. Our first main result is the following:
\begin{theorem}\label{thm:main}
Let $k$ be a perfect field and $R$ a commutative $\bbQ$-algebra. Given a mixed realization $H$, there exists a functor $H^{\mathrm{nc}}$ making the diagram commute:
\begin{equation}\label{eq:diagram-0}
\xymatrix{
\mathrm{Sm}(k) \ar[d]_-{X \mapsto \perf_\dg(X)} \ar[rr]^-{\mathrm{H}} &&  \mathrm{Mod}(\oplus_m H(\bbT)^{\otimes m}) \\
\Hmo(k) \ar@/_1.0pc/[urr]_-{H^{\mathrm{nc}}} &&\,.
}
\end{equation}
When $k$ is of characteristic zero, the same holds with $\mathrm{Sm}(k)$ replaced by $\mathrm{Sch}(k)$.
\end{theorem}
Intuitively speaking, Theorem \ref{thm:main} shows that as soon as we $\otimes$-trivialize the image of the Tate motive $H(\bbT)$, the modified mixed realization $\mathrm{H}$ factors through perfect complexes! This result is inspired by Kontsevich's definition of noncommutative {\'e}tale cohomology theory; consult the notes \cite{Kontsevich-talk}. 
\begin{corollary}[Derived Morita invariance]\label{cor:main}
Let $X$ and $Y$ be two smooth $k$-schemes of finite type and $H$ a mixed realization. If $\perf_\dg(X) \simeq \perf_\dg(Y)$ in $\Hmo(k)$, then $\mathrm{H}(X)\simeq \mathrm{H}(Y)$ in $\mathrm{Mod}(\oplus_m H(\bbT)^{\otimes m})$. When $k$ is of characteristic zero, the same holds without the smoothness assumption.
\end{corollary}
\section{Examples of modified mixed realizations}\label{sec:ex}
Let $R$ be a field extension of $\bbQ$ and $(\cC,\otimes, {\bf 1})$ an $R$-linear neutral Tannakian category equipped with a $\otimes$-invertible ``Tate'' object ${\bf 1}(1)$. In what follows, we write $\mathrm{Gal}(\cC)$ for the Tannakian group of $\cC$ and $\mathrm{Gal}_0(\cC)$ for the kernel of the homomorphism $\mathrm{Gal}(\cC) \twoheadrightarrow \bbG_m$, where $\bbG_m$ is the Tannakian group of the smallest Tannakian subcategory of $\cC$ containing ${\bf 1}(1)$. 

Let $H\colon \mathrm{DM}(k;R) \to \cD(\mathrm{Ind}(\cC))$ be an $R$-linear triangulated $\otimes$-functor with values in the derived category of ind-objects of $\cC$. Assume that $H$ preserves arbitrary direct sums and sends $R(1)$ to ${\bf 1}(1)$. Given such a functor, let $H^\ast$ be its composition with the total cohomology functor $\cD(\mathrm{Ind}(\cC)) \to \mathrm{Gr}_\bbZ(\mathrm{Ind}(\cC))$. Note that $H^\ast(\bbT)=H^2(\bbP^1)^{\otimes (-1)}$ and that $H$ and $H^\ast$ are mixed realizations.

Recall from the Tannakian formalism that, since $\cC$ is an $R$-linear neutral Tannakian category, $\mathrm{Gr}^b_\bbZ(\cC)$ is $\otimes$-equivalent to the $R$-linear category $\mathrm{Rep}_\bbZ(\mathrm{Gal}(\cC))$ of finite dimensional $\bbZ$-graded continuous representations of $\mathrm{Gal}(\cC)$. Recall also that the inclusion $\mathrm{Gal}_0(\cC) \subset \mathrm{Gal}(\cC)$ gives rise to the following restriction functor
\begin{eqnarray}\label{eq:restriction}
& \mathrm{Rep}_\bbZ(\mathrm{Gal}(\cC)) \too \mathrm{Rep}_{\bbZ/2}(\mathrm{Gal}_0(\cC)) & \{V_n\}_{n \in \bbZ} \mapsto (\oplus_{n} V_{2n}, \oplus_n V_{2n+1})\,,
\end{eqnarray}
where $\mathrm{Rep}_{\bbZ/2}(\mathrm{Gal}_0(\cC))$ stands for the category of finite dimensional $\bbZ/2$-graded continuous representations of $\mathrm{Gal}_0(\cC)$. Our second main result is the following:
\begin{theorem}\label{thm:main2}
Under the above assumptions, the restriction of the  base-change functor $\mathrm{Gr}_\bbZ(\mathrm{Ind}(\cC)) \stackrel{\gamma}{\to}\mathrm{Mod}(\oplus_m H^2(\bbP^1)^{\otimes (-m)})$ to $\mathrm{Gr}^b_\bbZ(\cC)$ admits a factorization:
$$ \mathrm{Gr}^b_\bbZ(\cC) \simeq \mathrm{Rep}_\bbZ(\mathrm{Gal}(\cC)) \stackrel{\eqref{eq:restriction}}{\too} \mathrm{Rep}_{\bbZ/2}(\mathrm{Gal}_0(\cC)) \subsetneq \mathrm{Mod}(\oplus_m H^2(\bbP^1)^{\otimes (-m)})\,.$$
Consequently, whenever the functor $H$ preserves compact objects, the modified mixed realization associated to $H^\ast$ is given by 
\begin{eqnarray*}
\mathrm{H}^\ast \colon \mathrm{Sm}(k) \too \mathrm{Rep}_{\bbZ/2}(\mathrm{Gal}_0(\cC)) && X \mapsto (\oplus_n H^{2n}(X), \oplus_n H^{2n+1}(X))\,.
\end{eqnarray*}
Moreover, when $k$ is of characteristic zero we can replace $\mathrm{Sm}(k)$ by $\mathrm{Sch}(k)$.
\end{theorem}
\smallbreak\noindent\textbf{Modified Nori realization.}
Let $k$ be a field of characteristic zero, equipped with an embedding $k \hookrightarrow \bbC$, and $R$ a field extension of $\bbQ$. Recall from \cite[\S2]{CG}\cite[\S8]{Huber} the construction of the $R$-linear neutral Tannakian category 
of Nori mixed motives $\mathrm{NMM}(k;R)$ and of its $\otimes$-invertible Tate object ${\bf 1}(1)$. As proved in \cite[Prop.~7.11]{CG}, there exists an $R$-linear triangulated $\otimes$-functor $H_N$ from $\mathrm{DM}(k;R)$ to $\cD(\mathrm{Ind}(\mathrm{NMM}(k;R)))$ which satisfies the conditions of Theorem \ref{thm:main2}. Consequently, the modified mixed realization associated to $H_N^\ast$ is given by
\begin{eqnarray*}
\mathrm{H}_N^\ast \colon \mathrm{Sch}(k) \too \mathrm{Rep}_{\bbZ/2}(\mathrm{Gal}_0(\mathrm{NMM}(k;R))) && X \mapsto (\oplus_n H_N^{2n}(X), \oplus_n H_N^{2n+1}(X))\,. 
\end{eqnarray*}
\smallbreak\noindent\textbf{Modified Jannsen realization.}
Recall from \cite[Part I]{Jannsen} the construction of the $R$-linear neutral Tannakian category of Jannsen mixed motives $\mathrm{JMM}(k;R)$ and of its Tate object ${\bf 1}(1)$. As explained in \cite[Prop.~10.3.3]{Huber}, the universal property of Nori's category of mixed motives yields an exact $\otimes$-functor from $\mathrm{NMM}(k;R)$ to $\mathrm{JMM}(k;R)$. The composition $H_J$ of $H_N$ with the functor from $\cD(\mathrm{Ind}(\mathrm{NMM}(k;R)))$ to $\cD(\mathrm{Ind}(\mathrm{JMM}(k;R)))$ satisfies the conditions of Theorem \ref{thm:main2}. Consequently, the modified mixed realization associated to $H^\ast_J$ is given by
\begin{eqnarray*}
\mathrm{H}_J^\ast \colon \mathrm{Sch}(k) \too \mathrm{Rep}_{\bbZ/2}(\mathrm{Gal}_0(\mathrm{JMM}(k;R))) && X \mapsto (\oplus_n H_J^{2n}(X), \oplus_n H_J^{2n+1}(X))\,. 
\end{eqnarray*}
\smallbreak\noindent\textbf{Modified de Rham realization.}
Let $\mathrm{Vect}(k)$ be the $k$-linear neutral Tannakian category of finite dimensional $k$-vector spaces, equipped with ${\bf 1}(1):=k$. In this case, the Tannakian group $\mathrm{Gal}_0(\mathrm{Vect}(k))$ is trivial and $\mathrm{Rep}_{\bbZ/2}(\mathrm{Gal}_0(\mathrm{Vect}(k)))$ reduces to the category of finite dimensional $\bbZ/2$-graded $k$-vector spaces $\mathrm{Vect}_{\bbZ/2}(k)$. Recall that $\mathrm{JMM}(k;\bbQ)$ comes equipped with an exact de Rham realization $\otimes$-functor from $\mathrm{JMM}(k;\bbQ)$ to $\mathrm{Vect}(k)$. The composition $H_{dR}$ of $H_J$ with the induced functor from $\cD(\mathrm{Ind}(\mathrm{JMM}(k;\bbQ)))$ to $\cD(\mathrm{Ind}(\mathrm{Vect}(k)))$ satisfies the conditions of Theorem \ref{thm:main2}. Consequently, the modified mixed realization associated to $H^\ast_{dR}$ is given by
\begin{eqnarray*}
\mathrm{H}_{dR}^\ast \colon \mathrm{Sch}(k) \too \mathrm{Vect}_{\bbZ/2}(k) && X \mapsto (\oplus_n H_{dR}^{2n}(X), \oplus_n H_{dR}^{2n+1}(X))\,. 
\end{eqnarray*}
\smallbreak\noindent\textbf{Modified Betti realization.}
Let $\mathrm{Vect}(\bbQ)$ be the $\bbQ$-linear neutral Tannakian category of finite dimensional $\bbQ$-vector spaces, equipped with ${\bf 1}(1):=\bbQ$. Recall that $\mathrm{JMM}(k;\bbQ)$ comes equipped with an exact Betti realization $\otimes$-functor from $\mathrm{JMM}(k;\bbQ)$ to $\mathrm{Vect}(\bbQ)$. The composition $H_{B}$ of $H_J$ with the induced functor from $\cD(\mathrm{Ind}(\mathrm{JMM}(k;\bbQ)))$ to $\cD(\mathrm{Ind}(\mathrm{Vect}(\bbQ)))$ satisfies the conditions of Theorem \ref{thm:main2}. Consequently, the modified mixed realization associated to $H^\ast_{B}$ is given by
\begin{eqnarray*}
\mathrm{H}_B^\ast \colon \mathrm{Sch}(k) \too \mathrm{Vect}_{\bbZ/2}(\bbQ) && X \mapsto (\oplus_n H_B^{2n}(X), \oplus_nH_B^{2n+1}(X))\,. 
\end{eqnarray*}
\smallbreak\noindent\textbf{Modified de Rham-Betti realization.}
Let $\mathrm{Vect}(k,\bbQ)$ be the $\bbQ$-linear neutral Tannakian category of triples $(V,W,\omega)$ (where $V$ is a finite dimensional $k$-vector space, $W$ a finite dimensional $\bbQ$-vector~space, and $\omega$ an isomorphism $V\otimes_k \bbC\to W\otimes_\bbQ \bbC$), equipped with the Tate object ${\bf 1}(1):=(k, \bbQ, \cdot (2\pi i)^{-1})$. Recall that $\mathrm{JMM}(k;\bbQ)$ comes equipped with an exact de Rham-Betti realization $\otimes$-functor from $\mathrm{JMM}(k;\bbQ)$ to $\mathrm{Vect}(k,\bbQ)$. The composition $H_{dRB}$ of $H_J$ with the functor from $\cD(\mathrm{Ind}(\mathrm{JMM}(k;\bbQ)))$ to $\cD(\mathrm{Ind}(\mathrm{Vect}(k,\bbQ)))$ satisfies the conditions of Theorem \ref{thm:main2}. Consequently, the modified mixed realization associated to $H^\ast_{dRB}$ is~given~by
\begin{eqnarray*}
\mathrm{H}_{dRB}^\ast \colon \mathrm{Sch}(k) \too \mathrm{Rep}_{\bbZ/2}(\mathrm{Gal}_0(\mathrm{Vect}(k,\bbQ))) & X \mapsto (\oplus_nH_{dRB}^{2n}(X), \oplus_n H_{dRB}^{2n+1}(X))\,. &
\end{eqnarray*}
\smallbreak\noindent\textbf{Modified {\'E}tale realization.}
Given a prime $l$, let $\mathrm{Rep}_l(\mathrm{Gal}(\overline{k}/k))$ be the $\bbQ_l$-linear neutral Tannakian category of finite dimensional $l$-adic representations of the absolute Galois group of $k$, equipped with the Tate object ${\bf 1}(1):=\mathrm{lim}_n \mu_{l^n}$. Recall that $\mathrm{JMM}(k;\bbQ)$ comes equipped with an exact {\'e}tale realization $\otimes$-functor from $\mathrm{JMM}(k;\bbQ)$ to $\mathrm{Rep}_l(\mathrm{Gal}(\overline{k}/k))$. The composition $H_{et}$ of $H_J$ with the functor from $\cD(\mathrm{Ind}(\mathrm{JMM}(k;\bbQ)))$ to $\cD(\mathrm{Ind}(\mathrm{Rep}_l(\mathrm{Gal}(\overline{k}/k))))$ satisfies the conditions of Theorem \ref{thm:main2}. Consequently, the modified mixed realization associated to $H_{et}^\ast$ is given by
\begin{eqnarray*}
\mathrm{H}_{et}^\ast \colon \mathrm{Sch}(k) \too \mathrm{Rep}_{\bbZ/2}(\mathrm{Gal}_0(\overline{k}/k)) && X \mapsto (\oplus_n H_{et}^{2n}(X), \oplus_n H_{et}^{2n+1}(X))\,.
\end{eqnarray*}
\begin{remark}
The preceding functor was suggested by Kontsevich in \cite{Kontsevich-talk}.
\end{remark}
\smallbreak\noindent\textbf{Modified Hodge realization.}
Recall from \cite[\S1]{Steenbrink} the construction of the $\bbQ$-linear neutral Tannakian category of mixed $\bbQ$-Hodge structures $\mathrm{MHS}(\bbQ)$ and of its $\otimes$-invertible Tate object ${\bf 1}(1)$. Recall that $\mathrm{JMM}(k;\bbQ)$ comes equipped with an exact Hodge realization $\otimes$-functor from $\mathrm{JMM}(k;\bbQ)$ to $\mathrm{MHS}(\bbQ)$. Let us denote by  The composition $H_{Hod}$ of $H_J$ with the induced functor from $\cD(\mathrm{Ind}(\mathrm{JMM}(k;\bbQ)))$ to $\cD(\mathrm{Ind}(\mathrm{MHS}(\bbQ)))$ satisfies the conditions of Theorem \ref{thm:main2}. Consequently, the modified mixed realization associated to $H_{Hod}$ is given by
\begin{eqnarray*}
\mathrm{H}_{Hod}^\ast \colon \mathrm{Sch}(k) \too \mathrm{Rep}_{\bbZ/2}(\mathrm{Gal}_0(\mathrm{MHS}(\bbQ))) && X \mapsto (\oplus_n H_{Hod}^{2n}(X), \oplus_n H_{Hod}^{2n+1}(X))\,.
\end{eqnarray*}
\begin{remark}[Modified pure $\bbR$-Hodge structures]
Recall from \cite[pages 33-34]{Steenbrink} the construction of the $\bbR$-linear neutral Tannakian category of pure $\bbR$-Hodge structures $\mathrm{HS}(\bbR)$ and of its $\otimes$-invertible Tate object ${\bf 1}(1)$. In this case, the Tannakian group $\mathrm{Gal}(\mathrm{HS}(\bbR))$ is the Hodge-Deligne circle $\mathrm{Res}_{\bbC/\bbR}(\bbG_m)$ and $\mathrm{Gal}_0(\mathrm{HS}(\bbR))$ the unitary group $U(1)$. Base-change along $\bbQ\subset\bbR$ gives then rise to the modified realization 
\begin{eqnarray*}
(H^\ast_{Hod})_\bbR\colon \mathrm{SmProj}(k) \too \mathrm{Rep}_{\bbZ/2}(U(1)) && X \mapsto (\oplus_n H_{Hod}^{2n}(X)_\bbR, \oplus_n H_{Hod}^{2n+1}(X)_\bbR)\,,
\end{eqnarray*}
where $\mathrm{SmProj}(k)$ stands for the category of smooth projective $k$-schemes.
\end{remark}
\section{New additive invariants}\label{sec:NCrealizations}
Let $\cA, \cB \subseteq \cC$ be dg categories yielding a semi-orthogonal decomposition $\dgHo(\cC)=\langle \dgHo(\cA), \dgHo(\cB)\rangle$ in the sense of Bondal-Orlov \cite{BO}. A functor $E\colon \Hmo(k) \to \cM$, with values in an additive category, is called an {\em additive invariant} if, for every dg categories $\cA, \cB \subseteq \cC$, the inclusions $\cA, \cB \subseteq \cC$ induce an isomorphism $E(\cA)\oplus E(\cB) \simeq E(\cC)$. Examples of additive invariants include algebraic $K$-theory, cyclic homology and all its variants, topological Hochschild homology, etc; consult \cite[\S2.2]{book}. As an application of Theorem \ref{thm:main}, we obtain several new examples of additive invariants:
\begin{proposition}\label{prop:new}
Given a mixed realization $H$, the associated functor $H^{\mathrm{nc}}$ (as in Theorem \ref{thm:main}) is an additive invariant. Moreover, the following holds:
\begin{itemize}
\item[(i)] Given a smooth $k$-scheme of finite type $Y$ and a smooth closed subscheme $X \hookrightarrow Y$, we have an isomorphism between $H^{\mathrm{nc}}(\perf_\dg(Y)_X)$ and $\mathrm{H}(X)$ where $\perf_\dg(Y)_X$ stands for the full dg subcategory of $\perf_\dg(Y)$ consisting of those perfect complexes which are supported on $X$;
\item[(ii)] Given a dg category $\cA$, we have $H^{\mathrm{nc}}(\cA[t]) \simeq H^{\mathrm{nc}}(\cA)$ where $\cA[t]:=\cA\otimes k[t]$.
\end{itemize}
\end{proposition}
By combining Proposition \ref{prop:new} with the modified mixed realizations of \S\ref{sec:ex}, we hence obtain the new additive invariants $H_N^{\mathrm{nc}}, H_J^{\mathrm{nc}}, H_{dR}^{\mathrm{nc}}, H_B^{\mathrm{nc}}, H_{dRB}^{\mathrm{nc}}, H_{et}^{\mathrm{nc}}, H_{Hod}^{\mathrm{nc}}$. Note that Theorem \ref{thm:main} determines the value of $H^{\mathrm{nc}}$ at the dg categories of the form $\perf_\dg(X)$. Moreover, Proposition \ref{prop:new}(i) shows that in order to compute $\mathrm{H}(X)$ we can first embed $X$ into {\em any} ambient smooth $k$-scheme $Y$ and then use the associated dg category $\perf_\dg(Y)_X$. In what follows, we compute the value of the new additive invariants $H^{\mathrm{nc}}$ at some ``truly noncommutative'' dg categories.
\smallbreak

\begin{example}[Finite dimensional algebras of finite global dimension]\label{ex:finite}
Let $A$ be a finite dimensional $k$-algebra of finite global dimension. We write $r$ for the number of simple (right) $A$-modules and $C_j$ for the center of the division $k$-algebra $\mathrm{End}_A(S_j)$ associated to the simple (right) $A$-module $S_j$. By combining Proposition \ref{prop:new} with \cite[pages 386-387]{Azumaya}, we obtain the computation $H^{\mathrm{nc}}(A) \simeq \oplus^r_{j=1} \mathrm{H}(\mathrm{Spec}(C_j))$. When $k$ is algebraically closed, we have $C_j=k$ and hence $H^{\mathrm{nc}}(A)\simeq \oplus^r_{j=1} \mathrm{H}(\mathrm{Spec}(k))$.
\end{example}
\begin{proposition}[Calkin algebra]\label{prop:Calkin}
%
Given a mixed realization $H$ as in Theorem \ref{thm:main2}, we have the following isomorphism of $\bbZ/2$-graded representations
$$ (H^\ast)^{\mathrm{nc}}(\Sigma(X)) \simeq (\oplus_n H^{2n+1}(X), \oplus_n H^{2n}(X)) \in \mathrm{Rep}_{\bbZ/2}(\mathrm{Gal}_0(\cC))\,,$$
where $\Sigma$ stands for the Calkin algebra and $\Sigma(X):=\perf_\dg(X) \otimes \Sigma$.
\end{proposition}
Roughly speaking, Proposition \ref{prop:Calkin} shows that the assignment $X \mapsto \Sigma(X)$ corresponds to switching the degrees of the $\bbZ/2$-graded representation $\mathrm{H}^\ast(X)$. Our third main result is the following computation:
\begin{theorem}[Differential operators]\label{thm:Weyl}
Let $k$ be a field of characteristic zero, $X$ a smooth $k$-scheme of finite type, and $\cD_X$ the sheaf of differential operators on $X$. Assume that there exists a filtration by closed subschemes
\begin{equation}\label{eq:filtration}
\emptyset = X_{-1} \hookrightarrow X_0 \hookrightarrow \cdots \hookrightarrow X_j \hookrightarrow \cdots \hookrightarrow X_{r-1} \hookrightarrow X_r = X
\end{equation}
such that $X_j \backslash X_{j-1}, 0 \leq j \leq r$, are smooth affine $k$-schemes of finite type\footnote{Thanks to the Bialynicki-Birula decomposition \cite{BB}, this holds, for example, for every smooth projective $k$-scheme $X$ equipped with a $\bbG_m$-action in which the fixed points are isolated (\eg\ projective homogeneous varieties, toric varieties, symmetric varieties, etc).}. Under these assumptions, $H^{\mathrm{nc}}(\perf_\dg(\cD_X))\simeq \mathrm{H}(X)$ for every mixed realization $H$.
\end{theorem}
%
\begin{example}[Weyl algebras]
In the particular case where $X=\bbA^r$, $\cD_X$ identifies with the $r^{\mathrm{th}}$ Weyl algebra $W_r$. Since the functor $\mathrm{H}$ is $\bbA^1$-homotopy invariant, it follows then from Theorem \ref{thm:Weyl} that $H^{\mathrm{nc}}(W_r) \simeq \mathrm{H}(\mathrm{Spec}(k))$.
\end{example}
\begin{example}[Lie algebras]
Let $G$ be a connected semisimple algebraic $\bbC$-group, $B$ a Borel subgroup of $G$, $\mathfrak{g}$ the Lie algebra of $G$, and $U_{\mathrm{ev}}(\mathfrak{g})/I$ the quotient of the universal enveloping algebra of $\mathfrak{g}$ by the kernel of the trivial character. Thanks to Beilinson-Bernstein's celebrated ``localisation'' result \cite{BeiBer}, it follows then from Theorem \ref{thm:Weyl} that $H^{\mathrm{nc}}(U_{\mathrm{ev}}(\mathfrak{g})/I) \simeq H^{\mathrm{nc}}(\perf_\dg(\cD_{G/B})) \simeq \mathrm{H}(G/B)$.
\end{example}
\begin{remark}
Theorem \ref{thm:Weyl} does {\em not} holds for every additive invariant! For example, in the case of Hochschild homology we have $HH_n(\perf_\dg(\cD_X))\simeq H_{dR}^{2d-n}(X)$ for every smooth affine $k$-scheme $X$ of dimension $d$; see \cite[Thm.~2]{Wodzicki}. Since $H_{dR}^{2d}(X)=~0$, this implies that $HH(\perf_\dg(\cD_X))\not\simeq HH(\perf_\dg(X))$. More generally, we have $HH(\perf_\dg(\cD_X))\not\simeq HH(A)$ for every {\em commutative} $k$-algebra $A$.
\end{remark}
\section{Periods of dg categories}\label{sec:periods}
Let $k$ be a field of characteristic zero, equipped with an embedding $k \hookrightarrow \bbC$, and $\bbC[t,t^{-1}]$ the $\bbZ$-graded $\bbC$-algebra of Laurent polynomials with $t$ of degree $1$.

Given an object $(V,W,\omega)$ of $\mathrm{Vect}(k,\bbQ)$, let $P(V,W,\omega)\subseteq \bbC$ be the subset of entries of the matrix representations of $\omega$ (with respect to basis of $V$ and $W$); see \cite[\S9.2]{Huber}. In the same vein, given an object $\{(V_n, W_n, \omega_n)\}_{n \in \bbZ}$ of $\mathrm{Gr}_\bbZ^b(\mathrm{Vect}(k,\bbQ))$, let $\cP(\{(V_n,W_n,\omega_n)\}_{n \in \bbZ})$ be the $\bbZ$-graded $k$-subalgebra of $\bbC[t,t^{-1}]$ generated in degree $n$ by the elements of the set $P(W_n,W_n, \omega_n)$. In the case of a $k$-scheme of finite type $X$, $\cP(X):=~\cP(H^\ast_{dRB}(X))$ is called the {\em $\bbZ$-graded algebra of periods of $X$}. This algebra, originally introduced by Grothendieck in the sixties, plays nowadays a key role in the study of transcendental numbers; see  \cite{Huber, KZ}.

Consider the $\bbZ/2$-graded $\bbC$-algebra $\bbC^{2\pi i}_{\bbZ/2}:=\bbC[t,t^{-1}]/\langle 1 - (2\pi i)t^2\rangle$ and the associated quotient homomorphism $\phi\colon \bbC[t,t^{-1}]\twoheadrightarrow \bbC^{2\pi i}_{\bbZ/2}$. 

The category $\mathrm{Ind}(\mathrm{Vect}(k,\bbQ))$ is equivalent to the category of triples $(V,W, \omega)$, where $V$ is a (not necessarily finite dimensional) $k$-vector space, $W$ a $\bbQ$-vector space, and $\omega$ an isomorphism $V\otimes_k \bbC \to W\otimes_\bbQ \bbC$. Given an object $\{(V_n, W_n, \omega_n)\}_{n \in \bbZ}$ of $\mathrm{Gr}_\bbZ(\mathrm{Ind}(\mathrm{Vect}(k,\bbQ)))$, let us denote by $\cP^{\mathrm{nc}}(\{(V_n,W_n,\omega_n)\}_{n \in \bbZ})$ the $\bbZ/2$-graded $k$-subalgebra of $\bbC^{2 \pi i}_{\bbZ/2}$ generated in degree $0$, resp. degree $1$, by the elements of~the~set
\begin{eqnarray*}
\cup_{n \in \bbZ} P(V_{2n},W_{2n},\omega_{2n})(2\pi i)^{-n} &&\mathrm{resp.}\,\,  \cup_{n \in \bbZ} P(V_{2n+1},W_{2n+1},\omega_{2n+1})(2\pi i)^{-n}\,.
\end{eqnarray*}
Making use of the new additive invariant $H^{\mathrm{nc}}_{dRB}$, we can now extend Grothendieck's theory of periods from schemes to the broad setting of dg categories.
\begin{definition}\label{def:periods}
Let $\cA$ be a dg category. The {\em $\bbZ/2$-graded algebra of periods $\cP^{\mathrm{nc}}(\cA)$ of $\cA$} is the $\bbZ/2$-graded $k$-algebra $\cP^{\mathrm{nc}}((H^\ast_{dRB})^{\mathrm{nc}}(\cA))$.
\end{definition}
Given dg categories $\cA$ and $\cB$, let $\cP^{\mathrm{nc}}(\cA) \diamond \cP^{\mathrm{nc}}(\cB)$ be the $\bbZ/2$-graded $k$-subalgebra of $\bbC^{2\pi i}_{\bbZ/2}$ generated by $\cP^{\mathrm{nc}}(\cA)$ and $\cP^{\mathrm{nc}}(\cB)$. Our fourth main result is the following:
\begin{theorem}\label{thm:main4}
The following implications hold:
\begin{itemize}
\item[(i)] If $\cA\simeq \cB$ in $\Hmo(k)$, then $\cP^{\mathrm{nc}}(\cA)=\cP^{\mathrm{nc}}(\cB)$;
\item[(ii)] If $(H^\ast_{dRB})^{\mathrm{nc}}(\cA)\simeq \mathrm{H}^\ast_{dRB}(X)$ for some $X \in \mathrm{Sch}(k)$, then $\cP^{\mathrm{nc}}(\cA)=\phi(\cP(X))$;
\item[(iii)] If $\cA, \cB \subseteq \cC$ are dg categories yielding a semi-orthogonal decomposition $\dgHo(\cC)=\langle \dgHo(\cA), \dgHo(\cB)\rangle$, then $\cP^{\mathrm{nc}}(\cC)=\cP^{\mathrm{nc}}(\cA) \diamond \cP^{\mathrm{nc}}(\cB)$.
\end{itemize}
\end{theorem}
\begin{corollary}[Derived Morita invariance]\label{cor:main4}
Let $X$ and $Y$ be two $k$-schemes of finite type. If $\perf_\dg(X) \simeq \perf_\dg(Y)$ in $\Hmo(k)$,~then $\phi(\cP(X))=\phi(\cP(Y))$.
\end{corollary}
Intuitively speaking, Theorem \ref{thm:main4} and Corollary \ref{cor:main4} show that as soon as we trivialize the graded polynomial $1 -(2\pi i)t^2 \in \bbC[t,t^{-1}]$, the resulting theory of periods factors through perfect complexes!
\begin{example}[Finite dimensional algebras of finite global dimension]
Let $A$ be a finite dimensional $k$-algebra of finite global dimension. Recall from Example \ref{ex:finite} that $(H_{dRB}^\ast)^{\mathrm{nc}}(A)\simeq \oplus^r_{j=1} \mathrm{H}^\ast_{dRB}(\mathrm{Spec}(C_j))$ for finite field extensions $C_j/k$. Theorem \ref{thm:main4} then implies that $\cP^{\mathrm{nc}}(A)=\phi(\cP(\Pi_{j=1}^r \mathrm{Spec}(C_j)))$. As explained in \cite[\S12.3]{Huber}, since $\mathrm{Spec}(C_j)$ is $0$-dimensional, $\cP(\Pi_{j=1}^r \mathrm{Spec}(C_j))$ agrees with the field extension $k(\cup_{j=1}^r C_j)$. Consequently, we conclude that $\cP^{\mathrm{nc}}(A)=k(\cup_{j=1}^r C_j)$. In the case where $k \subseteq \overline{\bbQ}$, we hence obtain solely algebraic numbers.
\end{example}

As the following example illustrates, in some cases Theorem \ref{thm:main4} furnishes ``noncommutative models'' for the algebra of periods:
\begin{example}[Quadric fibrations]\label{ex:fibrations}
Let $q\colon Q \to S$ be a flat quadric fibration of relative dimension $d$. As proved in \cite[Thm.~4.2]{KuznetsovFib}, we have a semi-orthogonal decomposition $\perf(Q)=\langle \perf(\cF), \perf(S)_0, \ldots, \perf_\dg(S)_{d-1}\rangle$, where $\cF$ stands for the sheaf of even parts of the Clifford algebra associated to $q$ and $\perf(S)_j\simeq \perf(S)$ for every $0 \leq j \leq d-1$. Assuming that $\perf(S)$ admits a full exceptional collection (\eg\ $S=\bbP^n$), we hence conclude from Theorem \ref{thm:main4} that
$$ \phi(\cP(Q))=\cP^{\mathrm{nc}}(\perf_\dg(\cF))\diamond \cP^{\mathrm{nc}}(k) \diamond \cdots \diamond \cP^{\mathrm{nc}}(k)=\cP^{\mathrm{nc}}(\perf_\dg(\cF))\,.$$
\end{example}
Roughly speaking, Example \ref{ex:fibrations} shows that modulo $2\pi i$ all the information about the periods of $Q$ is encoded in the ``noncommutative model'' $\perf_\dg(\cF)$.

As a further application of Theorem \ref{thm:main4}, we have the following homological projective duality (=HPD) invariance result: let $X$ be a smooth projective $k$-scheme equipped with an ample line bundle $\cO_X(1)$ and $X \to \bbP(V)$ the associated morphism where $V=H^0(X,\cO_X(1))^\ast$. Assume that $\perf(X)$ admits a Lefschetz decomposition $\perf(X)=\langle \bbA_0, \bbA_1(1), \ldots, \bbA_{i-1}(i-1)\rangle$ with respect to $\cO_X(1)$ in the sense of \cite[Def.~4.1]{KuznetsovHPD}. Following \cite[Def.~6.1]{KuznetsovHPD}, let $Y$ be the HP-dual\footnote{In general, the HP-dual of $X$ is a noncommutative variety in the sense of \cite[\S2.4]{ICM-Kuznetsov}.} of $X$ and $\cO_Y(1)$ and $Y\to \bbP(Y^\ast)$ the associated ample line bundle and morphism, respectively.
\begin{theorem}[HPD-invariance]\label{thm:HPD}
Let $X$ and $Y$ be as above. Given a linear subspace $L \subset V^\ast$, consider the associated linear sections\footnote{These linear sections are not necessarily smooth.} $X_L:=X\times_{\bbP(V)} \bbP(L^\perp)$ and $Y_L:=Y \times_{\bbP(Y^\ast)}~\bbP(L)$. Assume that the triangulated category $\bbA_0$ is generated by exceptional objects, that $\mathrm{dim}(X_L)=\mathrm{dim}(X) - \mathrm{dim}(L)$, and that $\mathrm{dim}(Y_L)=\mathrm{dim}(Y) - \mathrm{dim}(L^\perp)$. Under these notations and assumptions, we have $\phi(\cP(X_L))=\phi(\cP(Y_L))$.
\end{theorem}
Intuitively speaking, Theorem \ref{thm:HPD} shows that modulo $2\pi i$ the algebra of periods is invariant under homological projective duality. To the best of the author's knowledge, this invariance result is new in the literature.
\begin{example}
The assumptions of Theorem \ref{thm:HPD} are known to hold in the case of linear duality, Veronese-Clifford duality, Grassmannian-Pfaffian duality, spinor duality, etc; consult \cite[\S4]{ICM-Kuznetsov}\cite[\S10-11]{Kuznetsov1} and the references therein. In the case of Grassmannian-Pfaffian duality, we hence conclude, for example, that $\phi(\cP(X_L))=\phi(\cP(Y_L))$ when $X_L$ is a $K3$ surface and $Y_L$ a Pfaffian cubic $4$-fold, when $X_L$ and $Y_L$ are two non-birational Calabi-Yau $3$-folds, when $X_L$ is a Fano $3$-fold and $Y_L$ a cubic $3$-fold, when $X_L$ is a Fano $4$-fold of index $1$ and $Y_L$ a surface of degree $42$, when $X_L$ is a Fano $5$-fold of index $2$ and $Y_L$ a curve of genus $43$, etc.
\end{example}
%
%
\section{Preliminaries}
Throughout the note $k$ will be a perfect field and $R$ a commutative $\bbQ$-algebra.
\subsection{Dg categories}\label{sub:dg}
Let $(\cC(k),\otimes, k)$ be the category of (cochain) complexes of $k$-vector spaces. A {\em dg category $\cA$} is a category enriched over $\cC(k)$ and a {\em dg functor $F\colon \cA \to \cB$} is a functor enriched over $\cC(k)$; consult the survey \cite{ICM-Keller}.

Let $\cA$ be a dg category. The opposite dg category $\cA^\op$ has the same objects as $\cA$ and $\cA^\op(x,y):=\cA(y,x)$. The category $\dgHo(\cA)$ has the same objects as $\cA$ and $\dgHo(\cA)(x,y):=H^0(\cA(x,y))$, where $H^0$ stands for $0^{\mathrm{th}}$ cohomology. A {\em right dg $\cA$-module} is a dg functor $\cA^\op \to \cC_\dg(k)$ with values in the dg category $\cC_\dg(k)$ of complexes of $k$-vector spaces. Let us denote by $\cC(\cA)$ the category of right dg $\cA$-modules. Following \cite[\S3.2]{ICM-Keller}, the {\em derived category $\cD(\cA)$ of $\cA$} is defined as the localization of $\cC(\cA)$ with respect to the objectwise quasi-isomorphisms. We write $\cD_c(\cA)$ for the triangulated subcategory of compact objects. 

A dg functor $F:\cA\to \cB$ is called a {\em derived Morita equivalence} if it induces an equivalence of categories $\cD(\cA) \simeq \cD(\cB)$; see \cite[\S4.6]{ICM-Keller}. As proved in \cite[Thm.~5.3]{Additive}, $\dgcat(k)$ admits a Quillen model structure whose weak equivalences are the derived Morita equivalences. Let us denote by $\Hmo(k)$ the associated homotopy category.

The {\em tensor product $\cA\otimes\cB$} of dg categories is defined as follows: the set of objects is the cartesian product and $(\cA\otimes\cB)((x,w),(y,z)):= \cA(x,y) \otimes \cB(w,z)$. As explained in \cite[\S2.3]{ICM-Keller}, this construction gives rise to a symmetric monoidal structure on $\dgcat(k)$, which descends to the homotopy category $\Hmo(k)$. A dg {\em $\cA\text{-}\cB$-bimodule $\mathrm{B}$} is a dg functor $\cA\otimes \cB^\op \to \cC_\dg(k)$ or equivalently a right dg $(\cA^\op \otimes \cB)$-module.

Foloowing Kontsevich \cite{Miami,finMot,IAS}, a dg category $\cA$ is called {\em smooth} if the dg $\cA\text{-}\cA$-bimodule $\cA \otimes \cA^\op \to \cC_\dg(k), (x,y) \mapsto \cA(y,x)$, belongs to $\cD_c(\cA^\op\otimes \cA)$ and {\em proper} if $\sum_n \mathrm{dim}\, H^n\cA(x,y)< \infty$ for any pair of objects $(x,y)$. Examples include finite dimensional $k$-algebras of finite global dimension (when $k$ is perfect) and dg categories of perfect complexes $\perf_\dg(Y)$ associated to smooth proper $k$-schemes~$Y$.
\subsection{Orbit categories}\label{sub:orbit}
Let $(\cM, \otimes, {\bf 1})$ be an $R$-linear symmetric monoidal additive category and $\cO \in \cM$ a $\otimes$-invertible object. The associated {\em orbit category} $\cM/_{\!-\otimes \cO}$ has the same objects as $\cM$ and morphisms $\Hom_{\cM/_{\!-\otimes \cO}}(a,b)$ defined by the direct sum $\oplus_{m \in \bbZ} \Hom_\cM(a, b \otimes \cO^{\otimes m})$. Given objects $a, b, c$ and morphisms
\begin{eqnarray*}
\mathrm{f}=\{f_m\}_{m \in \bbZ} \in \oplus_{m} \Hom_\cM(a,b \otimes \cO^{\otimes m})&\mathrm{g}=\{g_m\}_{m \in \bbZ} \in \oplus_{m} \Hom_\cM(b,c \otimes \cO^{\otimes m})\,, &
\end{eqnarray*}
the $j^{\mathrm{th}}$-component of $\mathrm{g}\circ \mathrm{f}$ is defined as $\sum_m (g_{j -m} \otimes \cO^{\otimes m})\circ f_m$. The functor $\pi\colon \cM \to \cM/_{\!-\otimes \cO}$ defined by $a \mapsto a$ and $f \mapsto \mathrm{f}=\{f_m\}_{m \in \bbZ}$, where $f_0=f$ and $f_m=0$ if $m\neq 0$, is endowed with an isomorphism $\pi \circ (-\otimes \cO) \Rightarrow \pi$ and is $2$-universal among all such functors. The category $\cM/_{\!-\otimes \cO}$ is $R$-linear, additive, and inherits from $\cM$ a symmetric monoidal structure making $\pi$ symmetric monoidal.
\section{Proof of Theorem \ref{thm:main}}\label{sec:proof}
Let $\mathrm{SH}(k)$ be the Morel-Voevodsky's stable $\bbA^1$-homotopy category of $(\bbP^1,\infty)$-spectra \cite{MV,Voevodsky-ICM} and $\mathrm{DA}(k;R)$ the $R$-coefficients variant introduced in \cite[\S4]{Ayoub}. Recall from {\em loc. cit.} that these categories are related by a triangulated $\otimes$-functor $(-)_R\colon \mathrm{SH}(k) \to \mathrm{DA}(k;R)$. As proved in \cite{RSO}, the $E^\infty$-ring spectrum $\mathrm{KGL} \in \mathrm{SH}(k)$ representing homotopy $K$-theory admits a strictly commutative model. Therefore, we can consider the closed symmetric monoidal Quillen model category $\mathrm{Mod}(\mathrm{KGL}_R)$ of $\mathrm{KGL}_R$-modules. Let us denote by $\Ho(\mathrm{Mod}(\mathrm{KGL}_R))$ the associated homotopy category. By construction, we have the following composition 
\begin{equation}\label{eq:composition-last}
\mathrm{Sm}(k) \stackrel{\Sigma^\infty(-_+)}{\too} \mathrm{DA}(k;R) \stackrel{-\wedge \mathrm{KGL}_R}{\too} \Ho(\mathrm{Mod}(\mathrm{KGL}_R))\,.
\end{equation}
\begin{lemma}\label{lem:new}
\begin{itemize}
\item[(i)] The triangulated category $\Ho(\mathrm{Mod}(\mathrm{KGL}_R))$ is compactly generated by the objects $\Sigma^\infty(Y_+)_R \wedge \mathrm{KGL}_R$ with $Y$ a smooth projective $k$-scheme. Moreover, these latter objects are strongly dualizable and self-dual.  
\item[(ii)] The objects $\Sigma^\infty(X_+)_R\wedge \mathrm{KGL}_R$, with $X \in \mathrm{Sm}(k)$, are strongly dualizable.
\end{itemize}
\end{lemma}
\begin{proof}
(i) As proved in \cite[Prop.~2.2.27-2]{Ayoub}, the triangulated category $\mathrm{DA}(k;R)$ is compactly generated by the objects $\Sigma^\infty(Y_+)_R(m)$ with $Y$ a smooth projective $k$-scheme and $m \in \bbZ$. Thanks to the periodicity isomorphism $\mathrm{KGL}_R \simeq \mathrm{KGL}_R(1)[2]$, we then conclude that the category $\Ho(\Mod(\mathrm{KGL}_R))$ is compactly generated by the objects $\Sigma^\infty(Y_+) \wedge \mathrm{KGL}_R$. The fact that these latter objects are strongly dualizable and self-dual is proved in \cite[Lem.~8.22]{Bridge}.

(ii) It is well-known that the strongly dualizable objects of a closed symmetric monoidal triangulated category are stable under distinguished triangles and direct summands. Therefore, since the objects $\Sigma^\infty(X_+)_R \wedge \mathrm{KGL}_R$, with $X \in \mathrm{Sm}(k)$, are compact, the proof follows now from item (i).
\end{proof}
Recall from \cite[\S2]{A1-homotopy}\cite[\S4]{Hopf} the construction of the category of noncommutative mixed motives $\mathrm{Mot}(k;R)$. By construction, this closed symmetric monoidal triangulated comes equipped with a $\otimes$-functor $U(-)_R\colon \Hmo(k) \to \mathrm{Mot}(k;R)$. Moreover, it is naturally enriched over the derived category $\cD(R)$; we denote this enrichment by $\uHom_{\cD(R)}(-,-)$. Given dg categories $\cA$ and $\cB$, with $\cA$ smooth and proper, recall from \cite[Prop.~4.4]{Hopf} that we have a natural isomorphism
\begin{equation}\label{eq:K-spectra}
\uHom_{\cD(R)}(U(\cA)_R,U(\cB)_R) \simeq KH(\cA^\op \otimes \cB) \wedge HR\,,
\end{equation}
where $KH(\cA^\op \otimes \cB)$ stands for the homotopy $K$-theory spectrum of $\cA^\op \otimes \cB$ and $HR$ for the Eilenberg-MacLane ring spectrum of $R$.
\begin{proposition}\label{prop:very-last}
There exists an $R$-linear fully-faithful triangulated $\otimes$-functor $\Phi$ making the following diagram commute
\begin{equation}\label{eq:diagram-1}
\xymatrix{
\mathrm{Sm}(k) \ar[d]_-{\eqref{eq:composition-last}} \ar[rrr]^-{X \mapsto \perf_\dg(X)} &&& \mathrm{Hmo}(k) \ar[dd]^-{U(-)_R} \\
\Ho(\mathrm{Mod}(\mathrm{KGL}_R)) \ar[d]_-{\uHom(-,\mathrm{KGL}_R)}&&& \\
\Ho(\mathrm{Mod}(\mathrm{KGL}_R)) \ar[rrr]_-{\Phi} &&& \mathrm{Mot}(k;R)\,,
}
\end{equation}
where $\uHom(-,-)$ stands for the internal Hom. The functor $\Phi$ preserves moreover arbitrary direct sums.
\end{proposition}
\begin{proof}
As proved in \cite[Cor.~2.5(ii)]{Bridge}, there exists an $R$-linear fully-faithful triangulated $\otimes$-functor $\Phi$ making the following diagram commute:
$$
\xymatrix{
\mathrm{Sm}(k) \ar[d]_-{\Sigma^\infty(-_+)_R} \ar[rrr]^-{X \mapsto \perf_\dg(X)} &&& \mathrm{Hmo}(k) \ar[d]^-{U(-)_R} \\
\mathrm{DA}(k;R) \ar[d]_-{-\wedge \mathrm{KGL}_R} &&&  \mathrm{Mot}(k;R) \ar[d]^-{\uHom(-,U(k)_R)}\\
\Ho(\mathrm{Mod}(\mathrm{KGL}_R)) \ar[rrr]_-{\Phi} &&& \mathrm{Mot}(k;R)\,.
}
$$
The functor $\Phi$ preserves moreover arbitrary direct sums. The proof will consist on ``moving'' the internal Hom from the right hand-side to the left-hand side.

Recall from \cite[\S7-8]{Bridge} the construction of the closed symmetric monoidal triangulated category $\mathrm{DA}(k;\Mot(k;R))$, of the $E^\infty$-object in $\mathrm{KGL}_{\mathrm{nc};R}$ in $\mathrm{DA}(k;\Mot(k;R))$, which admits a strictly commutative model, of the closed symmetric monoidal Quillen model category $\mathrm{Mod}(\mathrm{KGL}_{\mathrm{nc};R})$ of $\mathrm{KGL}_{\mathrm{nc};R}$-modules, and of the homotopy category $\Ho(\mathrm{Mod}(\mathrm{KGL}_{\mathrm{nc};R}))$. By construction, we have an adjunction of categories
\begin{equation}\label{eq:adjunction-last-last}
\xymatrix{
\Ho(\mathrm{Mod}(\mathrm{KGL}_{\mathrm{nc};R})) \ar@<1ex>[d]^-{\mathrm{forget}} \\
\mathrm{DA}(k;\Mot(k;R)) \ar@<1ex>[u]^-{-\otimes \mathrm{KGL}_{\mathrm{nc};R}} \,.
}
\end{equation}
Recall also from \cite[page 535]{Voevodsky} that we have also a commutative diagram
\begin{equation}\label{eq:com-diagram}
\xymatrix{
\mathrm{DA}(k;\Mot(k;R)) \ar[rr]^-{-\otimes \mathrm{KGL}_{\mathrm{nc};R}} && \Ho(\mathrm{Mod}(\mathrm{KGL}_{\mathrm{nc};R})) \\
\mathrm{DA}(k;R) \ar[rr]_-{-\wedge \mathrm{KGL}_R} \ar[u] && \Ho(\mathrm{Mod}(\mathrm{KGL}_R)) \ar[u] \,,
}
\end{equation}
where the vertical left hand-side functor is induced by the change of coefficients from $R$ to $\Mot(k;R)$ and the vertical right hand-side functor is obtained by left Kan extension. By construction, the categories $\mathrm{DA}(k;\Mot(k;R))$ and $\Ho(\mathrm{Mod}(\mathrm{KGL}_{\mathrm{nc};R}))$ are enriched over noncommutative mixed motives $\mathrm{Mot}(k;R)$; we denote this enrichment by $\uHom_{\mathrm{Mot}(k;R)}(-,-)$. As explained in \cite[\S8]{Bridge}, $\Phi$ is defined as the composition of the vertical right hand-side functor in \eqref{eq:com-diagram} with $\uHom_{\mathrm{Mot}(k;R)}(\mathrm{KGL}_{\mathrm{nc};R},-)$.
We now claim that the following two compositions (of the diagram \eqref{eq:diagram-1})
$$
\xymatrix{
\mathrm{Sm}(k) \ar[rr]^-{X\mapsto \perf_\dg(X)}&&  \Hmo(k) \ar[r]^-{U(-)_R} & \Mot(k;R)}
$$
$$ 
\xymatrix{
\mathrm{Sm}(k) \ar[r]^-{\eqref{eq:composition-last}} & \Ho(\Mod(\mathrm{KGL}_R)) \ar[rr]^-{\uHom(-,\mathrm{KGL}_R)} && \Ho(\Mod(\mathrm{KGL}_R)) \ar[r]^-{\Phi} & \Mot(k;R)}
$$
can be described by the same formula:
\begin{equation}\label{eq:description}
X \mapsto \uHom_{\Mot(k;R)}(\mathrm{KGL}_{\mathrm{nc};R}, \uHom(\Sigma^\infty(X_+) \otimes \mathrm{KGL}_{\mathrm{nc};R}, \mathrm{KGL}_{\mathrm{nc};R}))\,.
\end{equation}
In what concerns the first composition, we have the following isomorphisms
\begin{eqnarray}
&& \uHom_{\Mot(k;R)}(\mathrm{KGL}_{\mathrm{nc};R}, \uHom(\Sigma^\infty(X_+) \otimes \mathrm{KGL}_{\mathrm{nc};R}, \mathrm{KGL}_{\mathrm{nc};R})) \label{eq:i-0} \\
& \simeq & \uHom_{\Mot(k;R)}(\Sigma^\infty(X_+) \otimes \mathrm{KGL}_{\mathrm{nc};R}, \mathrm{KGL}_{\mathrm{nc};R})) \label{eq:i-1} \\
& \simeq & \uHom_{\Mot(k;R)}(\Sigma^\infty(X_+), \mathrm{KGL}_{\mathrm{nc};R})) \label{eq:i-2} \\
& \simeq & U(\perf_\dg(X))_R \label{eq:i-3}\,,
\end{eqnarray}
where \eqref{eq:i-1} follows from the classical $(\otimes,\uHom)$ adjunction, \eqref{eq:i-2} from the above adjunction \eqref{eq:adjunction-last-last}, and \eqref{eq:i-3} from \cite[Prop.~7.26]{Voevodsky}. In what concerns the second composition, the noncommutative mixed motive \eqref{eq:i-0} identifies with 
$$\uHom_{\Mot(k;R)}(\mathrm{KGL}_{\mathrm{nc};R}, \uHom(\Sigma^\infty(X_+) \otimes \mathrm{KGL}_R, \mathrm{KGL}_R))$$
since the object $\Sigma^\infty(X_+)_R \wedge \mathrm{KGL}_R$ is strongly dualizable (see Lemma \ref{lem:new}(ii)), the vertical right hand-side functor in \eqref{eq:com-diagram} is symmetric monoidal, and the diagram \eqref{eq:com-diagram} is commutative. This finishes the proof of Proposition \ref{prop:very-last}.
\end{proof}
Let $\mathrm{HZ} \in \mathrm{SH}(k)$ be the $E^\infty$-ring spectrum representing motivic cohomology. On the one hand, we have $\mathrm{KGL}_R \simeq \oplus_{m \in \bbZ} \mathrm{HZ}_R(m)[2m]$; see \cite{Bloch}\cite[\S6]{Riou}. On the other hand, Voevodsky's big category of mixed motives $\mathrm{DM}(k;R)$ identifies with the homotopy category $\Ho(\mathrm{Mod}(\mathrm{HZ}_R))$ of $\mathrm{HZ}_R$-modules; see \cite{RO1}. Under this identification, the Tate object $\bbT:=R(1)[2]$ corresponds to the $\mathrm{HZ}_R$-module $\mathrm{HZ}_R(1)[2]$. Since the motives $M(X)_R$, with $X \in \mathrm{Sm}(k)$, are strongly dualizable, base-change along $\mathrm{HZ}_R \to \mathrm{KGL}_R$ gives then rise to an $R$-linear triangulated $\otimes$-functor $-\wedge_{\mathrm{HZ}_R} \mathrm{KGL}_R$ making the following diagram commute:
\begin{equation}\label{eq:diagram-2}
\xymatrix{
\mathrm{Sm}(k) \ar[d]_-{M(-)_R} \ar@{=}[rrr] &&& \mathrm{Sm}(k) \ar[d]^-{\eqref{eq:composition-last}} \\
\mathrm{DM}(k;R) \ar[d]_-{\uHom(-,M(\mathrm{Spec}(k))_R)} &&& \Ho(\mathrm{Mod}(\mathrm{KGL}_R)) \ar[d]^-{\uHom(-,\mathrm{KGL}_R)} \\
\mathrm{DM}(k;R) \ar[rrr]_-{-\wedge_{\mathrm{HZ}_R}\mathrm{KGL}_R} &&& \Ho(\mathrm{Mod}(\mathrm{KGL}_R)) \,.
}
\end{equation}
We now have all the ingredients necessary for the construction of the functor $H^{\mathrm{nc}}$. 

Thanks to Lemma \ref{lem:new}, the triangulated category $\Ho(\mathrm{Mod}(\mathrm{KGL}_R))$ is compactly generated. Therefore, since the triangulated functor $\Phi$ preserves arbitrary direct sums, we conclude from \cite[Thm.~8.4.4]{Neeman} that it admits a right adjoint~$\Phi^r$.

As mentioned above, we have $\mathrm{KGL}_R \simeq \oplus_{m \in \bbZ} \mathrm{HZ}_R(m)[2m]$. Therefore, since by assumption $H$ is lax $\otimes$-functor and $H(\oplus_m \bbT^{\otimes m})\simeq \oplus_m H(\bbT)^{\otimes m}$, we have the following commutative diagram
\begin{equation}\label{eq:diagram-3}
\xymatrix{
\mathrm{DM}(k;R) \ar@{=}[r] \ar[d]_-{-\wedge_{\mathrm{HZ}_R}\mathrm{KGL}_R}  & \mathrm{DM}(k;R) \ar[d]^-\gamma  \ar[r]^-H & \cM \ar[d]^-\gamma \\
\Ho(\mathrm{Mod}(\mathrm{KGL}_R))\ar[r]_-{\iota} & \mathrm{Mod}(\oplus_m \bbT^{\otimes m}) \ar[r]_-{H'} & \mathrm{Mod}(\oplus_m H(\bbT)^{\otimes m})\,,
}
\end{equation}
where $\iota$ stands for the canonical functor and $H'$ for the $R$-linear lax $\otimes$-functor naturally associated to $H$. The searched functor $H^{\mathrm{nc}}$ can now be defined as:
\begin{equation}\label{eq:composition}
\Hmo(k) \stackrel{U(-)_R}{\to} \mathrm{Mot}(k;R) \stackrel{\Phi^r}{\to} \Ho(\mathrm{Mod}(\mathrm{KGL}_R)) \stackrel{H'\circ  \iota}{\too} \mathrm{Mod}(\oplus_m H(\bbT)^{\otimes m})\,.
\end{equation}
The commutativity of diagram \eqref{eq:diagram-0} follows now from the commutativity of diagrams \eqref{eq:diagram-1} and \eqref{eq:diagram-2}-\eqref{eq:diagram-3}, and from the fact that the functor $\Phi$ is fully-faithful. This concludes the proof of the first claim of Theorem \ref{thm:main}. 

Let us now assume that $k$ is of characteristic zero and prove the second claim of Theorem \ref{thm:main}. Recall that a cartesian square of $k$-schemes of finite type
\begin{equation}\label{eq:square}
\xymatrix{
Z\times_X V \ar[d] \ar[r] & V \ar[d]^-p \\
Z \ar[r]_-i & X
}
\end{equation}
is called an {\em abstract blow-up square} if $i$ is a closed embedding and $p$ a proper map inducing an isomorphism $p^{-1}(X\backslash Z)_{\mathrm{red}}\simeq (X\backslash Z)_{\mathrm{red}}$. Consider the composition
$$
\Gamma_1\colon \mathrm{Sch}(k)\stackrel{\eqref{eq:composition-last}}{\too} \Ho(\Mod(\mathrm{KGL}_R)) \stackrel{\uHom(-, \mathrm{KGL}_R)}{\too} \Ho(\Mod(\mathrm{KGL}_R))
$$
as well as the composition
$$ \Gamma_2\colon \mathrm{Sch}(k) \stackrel{X \mapsto \perf_\dg(X)}{\too} \Hmo(k) \stackrel{U(-)_R}{\too} \Mot(k;R) \stackrel{\Phi^r}{\too} \Ho(\mathrm{Mod}(\mathrm{KGL}_R))\,.$$
Given a $k$-scheme of finite type $X$, Hironaka's resolution of singularities (see \cite[Thm.~1]{Hironaka}) yields a finite sequence of proper maps
$$ X_r \stackrel{p_r}{\too} X_{r-1} \too \cdots \too X_j \stackrel{p_j}{\too} X_{j-1} \too \cdots \too X_1 \stackrel{p_1}{\too} X_0:=X$$
with $X_r$ smooth and $X_j$ obtained from $X_{j-1}$ by an abstract blow-up square
\begin{equation*}
\xymatrix{
Z_{j-1}\times_{X_{j-1}} X_j \ar[d] \ar[r] & X_j \ar[d]^-{p_j} \\
Z_{j-1} \ar[r] & X_{j-1}
}
\end{equation*}
with $Z_{j-1}$ smooth. Using the commutativity of diagram \eqref{eq:diagram-1}, the fully-faithfulness of the functor $\Phi$, and the fact that the $k$-schemes $X_r$, $Z_{r-1}$, and $Z_{r-1} \times_{X_{r-1}} X_r$ are smooth, we can then inductively apply Lemma \ref{lem:key} below in order to conclude that $\Gamma_1(X)\simeq \Gamma_2(X)$. Consequently, by definition of $\mathrm{H}$ and $H^{\mathrm{nc}}$, we have $\mathrm{H}(X) \simeq H^{\mathrm{nc}}(\perf_\dg(X))$ in $\mathrm{Mod}(\oplus_m H(\bbT)^{\otimes m})$. This finishes the~proof of~Theorem~\ref{thm:main}.
\begin{lemma}\label{lem:key}
The following commutative squares
\begin{equation}\label{eq:squares}
\xymatrix{
\Gamma_1(X) \ar[r]^-{\Gamma_1(p)} \ar[d]_-{\Gamma_1(i)}& \Gamma_1(V) \ar[d] & \Gamma_2(X) \ar[r]^-{\Gamma_2(p)} \ar[d]_-{\Gamma_2(i)}& \Gamma_2(V) \ar[d]  \\
\Gamma_1(Z) \ar[r] & \Gamma_1(Z\times_X V) & \Gamma_2(Z) \ar[r] & \Gamma_1(Z\times_X V)\,,
}
\end{equation}
obtained by applying $\Gamma_1$ and $\Gamma_2$ to \eqref{eq:square}, are homotopy cartesian.
\end{lemma}
\begin{proof}
As explained in \cite[\S4]{JPAA}, the functor $\Sigma^\infty(-_+)$ satisfies {\em descent along abstract blow-up squares}, \ie it sends abstract blow-up squares to homotopy cartesian squares. Moreover, the objects $\Sigma^\infty(X_+)$, with $X \in \mathrm{Sch}(k)$, are compact in $\mathrm{SH}(k)$. Making use of Lemma \ref{lem:new}, we then conclude that the objects $\Sigma^\infty(X_+)_R \wedge \mathrm{KGL}_R$, with $X \in \mathrm{Sch}(k)$, are strongly dualizable in $\Ho(\Mod(\mathrm{KGL}_R))$. This implies that the left-hand side square in \eqref{eq:squares} is homotopy cartesian. 

Given $k$-schemes of finite type $X$ and $Y$, with $Y$ moreover smooth projective, we have natural isomorphisms
\begin{eqnarray}
&& \uHom_{\cD(R)}(U(\perf_\dg(Y))_R, U(\perf_\dg(X))_R) \nonumber\\
& \simeq & KH(\perf_\dg(Y)^\op \otimes \perf_\dg(X))\wedge HR \label{eq:star11}  \\
& \simeq  & KH(\perf_\dg(Y) \otimes \perf_\dg(X))\wedge HR \label{eq:star22} \\
& \simeq  & KH(Y \times X)\wedge HR  \label{eq:star33}\,,
\end{eqnarray}
where \eqref{eq:star11} is a particular case of \eqref{eq:K-spectra}, \eqref{eq:star22} follows from the derived Morita equivalence $\perf_\dg(Y)^\op \simeq \perf_\dg(Y), \cG \mapsto \uHom(\cG,\cO_Y)$, and \eqref{eq:star33} from the derived Morita equivalence $\perf_\dg(Y) \otimes \perf_\dg(X)\simeq \perf_\dg(Y \times X), (\cG,\cH) \mapsto \cG \boxtimes \cH$, proved in \cite[Lem.~4.26]{Gysin}. By applying the functor $\uHom_{\cD(R)}(U(\perf_\dg(Y))_R,-)$ to
$$
\xymatrix{
\perf_\dg(X) \ar[d]_-{i^\ast} \ar[r]^-{p^\ast} & \perf_\dg(V) \ar[d] \\
\perf_\dg(Z) \ar[r] & \perf_\dg(Z\times_X V) \,,
}
$$
we hence obtain the following commutative square:
\begin{equation}\label{eq:square2}
\xymatrix{
KH(Y \times X) \wedge HR \ar[r]^-{p^\ast} \ar[d]_-{i^\ast} & KH(Y \times V) \wedge HR \ar[d] \\
KH(Y \times Z) \wedge HR \ar[r] & KH(Y \times (Z\times_X V)) \wedge HR \,. 
}
\end{equation}
As proved in \cite[Thm.~3.5]{Haesemeyer}, homotopy $K$-theory satisfies descent along abstract blow-up squares. Therefore, since $-\wedge HR$ preserves homotopy (co)cartesian squares and $Y \times \eqref{eq:square}$ is also an abstract blow-up square, the preceding square \eqref{eq:square2} is homotopy cartesian. Lemma \ref{lem:new}, combined with the commutative diagram \eqref{eq:diagram-1}, allows us then to conclude that the right-hand side square in \eqref{eq:squares} is homotopy cartesian. This finishes the proof of Lemma \ref{lem:key}.
%
%
\end{proof}
\section{Proof of Theorem \ref{thm:main2}}
Note that $\oplus_m H^2(\bbP^1)^{\otimes (-m)}$ belongs to $\mathrm{Gr}_\bbZ(\mathrm{Ind}(\cC))$, that $H^2(\bbP^1)^{\otimes (-1)}$ belongs to $\mathrm{Gr}^b_\bbZ(\cC)$, and that we have the following adjunction of categories:
\begin{equation}\label{eq:adjunction}
\xymatrix{
\mathrm{Mod}(\oplus_m H^2(\bbP^1)^{\otimes (-m)}) \ar@<1ex>[d]^-{\mathrm{forget}} \\
\mathrm{Gr}_\bbZ(\mathrm{Ind}(\cC)) \ar@<1ex>[u]^-{\gamma} \,.
}
\end{equation}
Given any object $\mathrm{a}:=\{a_n\}_{n \in \bbZ}$ of $\mathrm{Gr}^b_\bbZ(\cC)$, we have a natural isomorphism 
\begin{eqnarray*}
\gamma(\mathrm{a}\otimes H^2(\bbP^1)^{\otimes (-1)})& := & a\otimes H^2(\bbP^1)^{\otimes (-1)} \otimes (\oplus_m H^2(\bbP^1)^{\otimes (-m)}) \\
& \simeq  & a \otimes (\oplus_m H^2(\bbP^1)^{\otimes (-m)}) =: \gamma(a)\,.
\end{eqnarray*}
Therefore, thanks to the universal property of orbit categories (see \S\ref{sub:orbit}), there exists an $R$-linear $\otimes$-functor $\gamma'$ making the diagram commute:
\begin{equation}\label{eq:diagram-iota}
\xymatrix{
\quad \quad \quad \quad  \quad \quad \mathrm{Gr}^b_\bbZ(\cC) \ar[d]_-\pi \subset \mathrm{Gr}_\bbZ(\mathrm{Ind}(\cC)) \ar[r]^-\gamma & \mathrm{Mod}(\oplus_m H^2(\bbP^1)^{\otimes (-m)}) \\
\mathrm{Gr}^b_\bbZ(\cC)/_{\!-\otimes H^2(\bbP^1)^{\otimes (-1)}} \ar@/_1.5pc/[ur]_-{\gamma'} &\,.
}
\end{equation}
Given objects $\mathrm{a}, \mathrm{b}$ of $\mathrm{Gr}^b_\bbZ(\cC)$, we have natural isomorphisms
\begin{eqnarray}
\Hom_{\mathrm{Mod}(\oplus_m H^2(\bbP^1)^{\otimes (-m)})}(\gamma(\mathrm{a}),\gamma(\mathrm{b})) & \stackrel{\eqref{eq:adjunction}}{\simeq} &
 \Hom_{\mathrm{Gr}_\bbZ(\mathrm{Ind}(\cC))}(\mathrm{a}, \mathrm{b} \otimes \oplus_m  H^2(\bbP^1)^{\otimes (-m)}) \nonumber \\
& \simeq & \oplus_{m \in \bbZ} \Hom_{\mathrm{Gr}^b_\bbZ(\cC)}(\mathrm{a}, \mathrm{b} \otimes H^2(\bbP^1)^{\otimes (-m)}) \label{eq:star2}\\
& = & \Hom_{\mathrm{Gr}^b_\bbZ(\cC)/_{\!-\otimes H^2(\bbP^1)^{\otimes (-1)}}}(\pi(\mathrm{a}), \pi(\mathrm{b}))  \nonumber \,,
\end{eqnarray}
where \eqref{eq:star2} follows from the fact that the functor $\mathrm{b}\otimes -$ preserves arbitrary direct sums and that the objects of $\mathrm{Gr}^b_\bbZ(\cC)$ are compact in $\mathrm{Gr}_\bbZ(\mathrm{Ind}(\cC))$. This implies that the functor $\gamma'$ in diagram \eqref{eq:diagram-iota} is moreover fully-faithful. 

Recall from the Tannakian formalism that, since $\cC$ is an $R$-linear neutral Tannakian category, $\mathrm{Gr}^b_\bbZ(\cC)$ is $\otimes$-equivalent to the $R$-linear category $\mathrm{Rep}(\mathrm{Gal}(\cC)\times \bbG_m)$ of finite dimensional continuous representations of $\mathrm{Gal}(\cC) \times \bbG_m$. The weight grading $\omega$, induced by the canonical morphism $\bbG_m \to \mathrm{Gal}(\cC) \times \bbG_m$, and the $\otimes$-invertible object $H^2(\bbP^1)^{\otimes (-1)}$ equip $\mathrm{Gr}_\bbZ^b(\cC)$ with a neutral {\em Tate triple} structure in the sense of Deligne-Milne \cite[\S5]{DM}. Therefore, as proved in \cite[Prop.~14.1]{JEMS}, the orbit category $\mathrm{Gr}^b_\bbZ(\cC)/_{\!-\otimes H^2(\bbP^1)^{\otimes (-1)}}$ becomes a neutral Tannakian category. Moreover, its Tannakian group is given by the kernel of the homomorphism 
\begin{equation}\label{eq:Galois}
\xymatrix@C=2em@R=2em{
\mathrm{Gal}(\mathrm{Gr}_\bbZ^b(\cC))= \mathrm{Gal}(\cC) \times \bbG_m \ar@{->>}[r]  &\bbG_m\,,}
\end{equation}
where the right-hand side copy of $\bbG_m$ is the Tannakian group of the smallest Tannakian subcategory of $\mathrm{Gr}_\bbZ^b(\cC)$ containing $H^2(\bbP^1)^{\otimes (-1)}$. Note that the first component of \eqref{eq:Galois} is the homomorphism $\mathrm{Gal}(\cC) \twoheadrightarrow \bbG_m$ introduced in \S\ref{sec:ex}, while the second component $\bbG_m \to \bbG_m$ is multiplication by $2$. This implies that the kernel of \eqref{eq:Galois} is equal to $\mathrm{Gal}_0(\cC) \times \mu_2$. Consequently, thanks to the Tannakian formalism, $\mathrm{Gr}^b_\bbZ(\cC)/_{\!-\otimes H^2(\bbP^1)^{\otimes (-1)}}$ is $\otimes$-equivalent to the $R$-linear category $\mathrm{Rep}(\mathrm{Gal}_0(\cC) \times \mu_2)$ of finite dimensional continuous representations of $\mathrm{Gal}_0(\cC) \times \mu_2$. Finally, under the $\otimes$-equivalences of categories between $\mathrm{Rep}(\mathrm{Gal}(\cC) \times \bbG_m)$ and $\mathrm{Rep}_\bbZ(\mathrm{Gal}(\cC))$ and $\mathrm{Rep}(\mathrm{Gal}_0(\cC) \times \mu_2)$ and $\mathrm{Rep}_{\bbZ/2}(\mathrm{Gal}_0(\cC))$, respectively, the functor $\pi$ in \eqref{eq:diagram-iota} identifies with the restriction functor \eqref{eq:restriction}. This concludes the proof of~Theorem~\ref{thm:main2}.
\section{Proof of Proposition \ref{prop:new}}
Recall from \eqref{eq:composition} the definition of the functor $H^{\mathrm{nc}}:=H'\circ \iota \circ \Phi^r \circ U(-)_R$. By construction, the functors $\Phi^r$, $H'$, and $\iota$, are additive. Therefore, since $U(-)_R$ is an additive invariant (see \cite[\S8.4.5]{book}), we conclude that $H^{\mathrm{nc}}$ is also an additive invariant. In what concerns item (i), we have an isomorphism between $U(\perf_\dg(Y)_X)_R$ and $U(\perf_\dg(X))_R$ in $\mathrm{Mot}(k;R)$. Hence, the proof follows from the definition of $H^{\mathrm{nc}}$ and from Theorem \ref{thm:main}. Item (ii) follows from the isomorphism $U(\cA[t])_R\simeq U(\cA)_R$ in $\Mot(k;R)$ (see \cite[\S2]{A1-homotopy}) and from the definition of $H^{\mathrm{nc}}$.
\section{Proof of Proposition \ref{prop:Calkin}}
As proved in \cite[Thm.~1.2]{Universal}, there is an isomorphism between $U(\Sigma(X))_R$ and $\Sigma(U(\perf_\dg(X))_R)$ in the triangulated category of noncommutative mixed motives $\Mot(k;R)$. Therefore, the proof follows from the definition of $(H^\ast)^{\mathrm{nc}}$.
\section{Proof of Theorem \ref{thm:Weyl}}
Consider the canonical dg functor $-\otimes_{\cO_X} \cD_X\colon \perf_\dg(X) \to \perf_\dg(\cD_X)$. The proof will consist on showing that the image of $U(\perf_\dg(X))_R \to U(\perf_\dg(\cD_X))_R$ under the functor $\Phi^r$ (see \S\ref{sec:proof}) is invertible. By construction of $H^{\mathrm{nc}}$, this implies that $H^{\mathrm{nc}}(\perf_\dg(\cD_X))\simeq \mathrm{H}(X)$ for every mixed realization $H$. 

Following Lemma \ref{lem:new}(i), the triangulated category $\Ho(\Mod(\mathrm{KGL}_R))$ is compactly generated by the strongly dualizable and self-dual objects $\Sigma^\infty(Y_+) \wedge \mathrm{KGL}_R$ with $Y$ a smooth projective $k$-scheme. Therefore, thanks to the commutative diagram \eqref{eq:diagram-1}, in order to show that the image of $U(\perf_\dg(X))_R \to U(\perf_\dg(\cD_X))_R$ under $\Phi^r$ is invertible, it suffices to show that the induced morphisms
$$ \uHom_{\cD(R)}(U(\perf_\dg(Y))_R, U(\perf_\dg(X))_R \too U(\perf_\dg(\cD_X))_R)$$
are invertible. As mentioned in \S\ref{sec:proof}, the preceding morphism identifies with
\begin{equation*}
KH(\perf_\dg(Y)^\op \otimes \perf_\dg(X))\wedge HR \too KH(\perf_\dg(Y)^\op \otimes \perf_\dg(\cD_X))\wedge HR\,.
\end{equation*}
Using the derived Morita equivalence $\perf_\dg(Y)^\op \simeq\perf_\dg(Y), \cG \mapsto \uHom(\cG,\cO_Y)$, we then conclude from Theorem \ref{thm:key2} below that the preceding morphism is invertible. This finishes the proof of Theorem \ref{thm:Weyl}. 
\begin{theorem}\label{thm:key2}
Given smooth $k$-schemes of finite type $X$ and $Y$, with $X$ as in Theorem \ref{thm:Weyl}, the induced morphism is invertible:
\begin{equation}\label{eq:induced-2}
KH(\perf_\dg(Y)\otimes \perf_\dg(X)) \too KH(\perf_\dg(Y) \otimes \perf_\dg(\cD_X))\,.
\end{equation}
\end{theorem}
The remainder of this section is devoted to the proof of Theorem \ref{thm:key2}. We will consider \eqref{eq:induced-2} as a morphism in two variables ($Y$ the first variable and $X$ the second variable) and divide the proof into three main steps: (i) reduction to the affine case in the first variable; (ii) reduction to the affine case in the second variable; (iii) proof of the affine case. We start with step (i). Given any dg functor $\cA \to \cB$, consider the induced morphism
$$ \alpha_Y\colon KH(\perf_\dg(Y)\otimes \cA) \too KH(\perf_\dg(Y) \otimes \cB)\,.$$
\begin{proposition}\label{prop:1}
If $\alpha_Y$ is invertible when $Y=\mathrm{Spec}(B)$ is a smooth affine scheme of finite type, then $\alpha_Y$ is invertible for every smooth scheme of finite type $Y$.
\end{proposition}
\begin{proof}
In order to simplify the exposition, let $KH(Y;\cA):=KH(\perf_\dg(Y)\otimes \cA)$; similarly with $\cA$ replaced by $\cB$. Given a Zariski open cover $U_1 \cup U_2 = Y'$ of a smooth $k$-scheme of finite type $Y'$, let us write $U_{12}$ for the intersection $U_1 \cap U_2$. Thanks to the induction principle \cite[Prop.~3.3.1]{BV}, in order to prove Proposition \ref{prop:1} it suffices to prove the following condition: if $\alpha_{U_1}$, $\alpha_{U_2}$, and $\alpha_{U_{12}}$ are invertible, then $\alpha_{Y'}$ is also invertible. Consider the commutative diagram:
\begin{equation}\label{eq:squares-big}
\xymatrix@C=1.5em@R=1.5em{
KH(Y'; \cB) \ar[ddd] \ar[rrr] & & & KH(U_1; \cB) \ar[ddd] \\
& KH(Y';\cA) \ar[ul]_-{\alpha_{Y'}} \ar[r] \ar[d] & KH(U_1;\cA) \ar[ur]^-{\alpha_{U_1}} \ar[d] & \\
& KH(U_2;\cA) \ar[r] \ar[dl]_-{\alpha_{U_2}} & KH(U_{12};\cA) \ar[dr]^-{\alpha_{U_{12}}} & \\
KH(U_2; \cB) \ar[rrr] & & & KH(U_{12}; \cB)\,.
}
\end{equation}
We claim that the ``front'' and ``back'' squares of \eqref{eq:squares-big} are homotopy cartesian. Note that this implies the preceding condition and consequently finishes the proof of Proposition \ref{prop:1}. We will focus ourselves solely in the ``back'' square; the proof of the other case is similar. Consider the following commutative diagram
\begin{equation}\label{eq:shorts}
\xymatrix{
\perf_\dg(Y')_Z \ar[d] \ar[r] & \perf_\dg(Y') \ar[d] \ar[r] & \perf_\dg(U_1) \ar[d] \\
\perf_\dg(U_2)_Z \ar[r] & \perf_\dg(U_2) \ar[r] & \perf_\dg(U_{12})
}
\end{equation}
in $\Hmo(k)$, where $Z$ stands for the closed complement $Y'\backslash U_1= U_2\backslash U_{12}$. As explained in \cite[\S5]{TT}, both rows of \eqref{eq:shorts} are short exact sequences of dg categories (see \cite[\S4.6]{ICM-Keller}). Moreover, the induced dg functor $\perf_\dg(Y')_Z\to \perf_\dg(U_2)_Z$ is a derived Morita equivalence; see \cite[Thm.~2.6.3]{TT}. As proved in \cite[Prop.~1.6.3]{Drinfeld}, the functor $-\otimes \cA$ preserves short exact sequences of dg categories. Therefore, \eqref{eq:shorts} gives rise to the commutative diagram
\begin{equation}\label{eq:shorts-1}
\xymatrix@C=1.5em@R=2.5em{
\perf_\dg(Y')_Z\otimes \cA \ar[d]_-{\simeq} \ar[r] & \perf_\dg(Y')\otimes \cA \ar[d] \ar[r] & \perf_\dg(U_1) \otimes \cA \ar[d] \\
\perf_\dg(U_2)_Z \otimes \cA \ar[r] & \perf_\dg(U_2) \otimes \cA \ar[r] & \perf_\dg(U_{12})\otimes \cA\,,
}
\end{equation} 
where both rows are short exact sequences of dg categories and the left vertical morphism is an isomorphism. Since homotopy $K$-theory sends short exact sequences of dg categories to homotopy cofiber sequences of spectra (see \cite[\S5.3]{A1-homotopy}), we then conclude from \eqref{eq:shorts-1} that the ``back'' square of \eqref{eq:squares-big} is homotopy cartesian.
\end{proof}
\begin{remark}\label{rk:new1}
\begin{itemize}
\item[(i)] By applying Proposition \ref{prop:1} to $\perf_\dg(X) \to \perf_\dg(\cD_X)$, we conclude that it suffices to prove Theorem \ref{thm:key2} in the particular case where $Y=\mathrm{Spec}(B)$ is a smooth affine $k$-scheme of finite type.
\item[(ii)] Proposition \ref{prop:1} holds {\em mutatis mutandis} without the smoothness assumption.
\end{itemize}
\end{remark}
We address now step (ii). Given a dg category $\cA$, consider the induced morphism
$$\beta_X\colon KH(\cA\otimes \perf_\dg(X)) \too KH(\cA\otimes \perf_\dg(\cD_X))\,.$$
\begin{proposition}\label{prop:2}
If $\beta_X$ is invertible when $X=\mathrm{Spec}(A)$ is a smooth affine scheme of finite type, then $\beta_X$ is invertible for every smooth scheme as in Theorem \ref{thm:Weyl}.
\end{proposition}
\begin{proof}
In order to simplify the exposition, let $KH(\cA;X):=KH(\cA\otimes \perf_\dg(X))$; similarly with $X$ replaced by $\cD_X$. Let $X'$ be a smooth $k$-scheme of finite type, $i\colon Z\hookrightarrow X'$ a smooth closed subscheme, and $j\colon U \hookrightarrow X'$ the open complement of $Z$. On the one hand, since homotopy $K$-theory is $\bbA^1$-homotopy invariant and sends short exact sequences of dg categories to homotopy cofiber sequences, \cite[Thm.~1.9 and Rk.~1.11(G2)]{Gysin} yields an homotopy cofiber sequence of spectra
$$KH(\cA;Z) \stackrel{i_\ast}{\too} KH(\cA;X') \stackrel{j^\ast}{\too} KH(\cA;U)\,.$$
On the other hand, we have a short exact sequence of dg categories (see \cite[\S3.1.4]{GD}):
$$ \perf_\dg(\cD_Z) \stackrel{i_{\mathrm{dR}, \ast}}{\too} \perf_\dg(\cD_{X'}) \stackrel{j^\ast}{\too} \perf_\dg(\cD_U)\,.$$
Hence, as in the proof of Proposition \ref{prop:1}, we obtain the homotopy cofiber sequence
$$ KH(\cA;\cD_Z) \stackrel{i_{\mathrm{dR},\ast}}{\too} KH(\cA;\cD_{X'}) \stackrel{j^\ast}{\too} KH(\cA;\cD_U)\,.$$
In the particular case where $X':=X\backslash X_{j-1}$ and $Z:=X_j\backslash X_{j-1}$, the above (general) considerations lead to the following commutative diagram:
\begin{equation}\label{eq:commutative-verylast}
\xymatrix{
KH(\cA;\cD_{X_j\backslash X_{j-1}}) \ar[r]^-{i_{dR,\ast}} & KH(\cA; \cD_{X\backslash X_{j-1}}) \ar[r]^-{j^\ast} & KH(\cA; \cD_{X\backslash X_j}) \\
KH(\cA; X_j \backslash X_{j-1}) \ar[u]^-{\beta_{X_j\backslash X_{j-1}}} \ar[r]_{i_\ast} & KH(\cA;X\backslash X_{j-1}) \ar[u]_-{\beta_{X\backslash X_{j-1}}} \ar[r]_-{j^\ast} & KH(\cA;X\backslash X_j) \ar[u]_-{\beta_{X\backslash X_j}}\,.
}
\end{equation}
Making use of the commutative diagrams \eqref{eq:commutative-verylast}, of the filtration \eqref{eq:filtration}, and of the fact that $X_j \backslash X_{j-1}, 0 \leq j \leq r$, are smooth affine $k$-schemes of finite type, the proof follows now from a descending induction argument on the index $j$.
\end{proof}
\begin{remark}\label{rk:new2}
By applying Proposition \ref{prop:2} to the dg category $\perf_\dg(X)$, we conclude that it suffices to prove Theorem \ref{thm:key2} in the particular case where $X=\mathrm{Spec}(A)$ is a smooth affine $k$-scheme of finite type.
\end{remark}
Finally, we address step (iii). Let $X=\mathrm{Spec}(A)$ and $Y=\mathrm{Spec}(B)$ be two smooth affine $k$-schemes of finite type. Thanks to Remarks \ref{rk:new1} and \ref{rk:new2}, the proof of Theorem \ref{thm:key2} follows now from the following result:
\begin{proposition}\label{prop:3}
The induced morphism $KH(B\otimes A) \to KH(B \otimes \cD_A)$, where $\cD_A$ stands for the $k$-algebra of differential operators on $A$, is invertible.
\end{proposition}
\begin{proof}
Let us denote by $T^\ast X =\mathrm{Spec}(C)$ be the cotangent bundle of $X$. Recall that we have an increasing filtration $0 = F_{-1} \cD_A \subset F_0 \cD_A \subset \cdots \subset F_j \cD_A \subset \cdots \subset \cD_A$ of $\cD_A$ given by the order of the differential operators. In particular, $F_0\cD_A=A$. This filtration is exhaustive, \ie $\cD_A=\cup_{j=-1}^\infty F_j \cD_A$, and the associated graded algebra $\mathrm{gr}(\cD_A)$ is isomorphic to $C$. Consequently, by applying the functor $B \otimes -$ to the preceding filtration we obtain an increasing exhaustive filtration 
$$ 0 = B\otimes F_{-1} \cD_A \subset B \otimes F_0 \cD_A \subset \cdots \subset B \otimes F_j\cD_A \subset \cdots \subset B \otimes \cD_A$$
of $B \otimes \cD_A$ with $F_0(B \otimes \cD_A)=B \otimes A$ and $\mathrm{gr}(B \otimes \cD_A)\simeq B \otimes C$.

Since the affine $k$-schemes $X$, $Y$, and $T^\ast X$, are smooth and of finite type, the $k$-algebras $A$, $B$, and $C$ are Noetherian and of finite global dimension. Consequently, the $k$-algebras $B\otimes A$ and $B\otimes C$ are also Noetherian and of finite global dimension. Making use of \cite[\S6 Thm.~7]{Quillen} (see also \cite[Thm.~1.1]{Hodge}), we then conclude that the morphism $K(B \otimes A) \to K(B \otimes \cD_A)$, induced by the inclusion $F_0(B\otimes \cD_A) \subset B \otimes \cD_A$, is invertible. Since the $k$-algebras $B\otimes A$ and $B \otimes \cD_A$ are regular Noetherian\footnote{The global dimension of $\cD_A$ is equal to the dimension of $X=\mathrm{Spec}(A)$.}, the proof follows now from the isomorphisms $K(B \otimes A) \simeq KH(B \otimes A)$ and $K(B \otimes \cD_A) \simeq KH(B \otimes \cD_A)$ between algebraic $K$-theory and homotopy $K$-theory.
\end{proof}
\section{Proof of Theorem \ref{thm:main4}}\label{sec:proof4}
If $\cA\simeq \cB$ in $\Hmo(k)$, then $(H^\ast_{dRB})^{\mathrm{nc}}(\cA)=(H^\ast_{dRB})^{\mathrm{nc}}(\cB)$. Therefore, item (i) follows from Definition \ref{def:periods}. Let us now prove item (ii). Recall from the proof of Theorem \ref{thm:main2} that we have the following adjunction of categories:
\begin{equation}\label{eq:adjunction-last}
\xymatrix{
\mathrm{Mod}(\oplus_m H_{dRB}^2(\bbP^1)^{\otimes (-m)}) \ar@<1ex>[d]^-{\mathrm{forget}} \\
\mathrm{Gr}_\bbZ(\mathrm{Ind}(\mathrm{Vect}(k,\bbQ))) \ar@<1ex>[u]^-{\gamma} \,.
}
\end{equation}
Recall also from \S\ref{sec:ex} that $H^\ast_{dRB}(X)$ belongs to the full subcategory $\mathrm{Gr}^b_\bbZ(\mathrm{Vect}(k,\bbQ))$. The proof of item (ii) is then a consequence of the following equalities
\begin{eqnarray}
\cP^{\mathrm{nc}}(\cA) & := & \cP^{\mathrm{nc}}(\mathrm{forget}((H^\ast_{dRB})^{\mathrm{nc}}(\cA)))\nonumber \\
& = & \cP^{\mathrm{nc}}(\mathrm{forget}(\mathrm{H}^\ast_{dRB}(X))) \label{eq:star-000} \\
& = & \cP^{\mathrm{nc}}(H^\ast_{dRB}(X)\otimes (\oplus_m H^2_{dRB}(\bbP^1)^{\otimes (-m)})) \label{eq:star-111} \\
& = & \cP^{\mathrm{nc}}(\oplus_m (H^\ast_{dRB}(X)\otimes H^2_{dRB}(\bbP^1)^{\otimes (-m)})) \label{eq:star-222} \\
& = & \cP^{\mathrm{nc}}(H^\ast_{dRB}(X)) \label{eq:star-333}\\\
& = & \phi(\cP(H^\ast_{dRB}(X))) = \phi(\cP(X))\,, \label{eq:star-444}
\end{eqnarray}
where \eqref{eq:star-000} follows from the assumption $(H^\ast_{dRB})^{\mathrm{nc}}(\cA)\simeq \mathrm{H}^\ast_{dRB}(X)$,  \eqref{eq:star-111} from~adjunction \eqref{eq:adjunction-last}, \eqref{eq:star-222} from the fact that $H^\ast_{dRB}(X)\otimes -$ preserves~arbitrary direct sums, \eqref{eq:star-333} from Lemma \ref{lem:key-last} below, and \eqref{eq:star-444} from Proposition \ref{prop:key-last-1}~below.

\begin{lemma}\label{lem:key-last}
For every object $\{(V_n,W_n,\omega_n)\}_{n \in \bbZ}$ of $\mathrm{Gr}^b_\bbZ(\mathrm{Vect}(k,\bbQ))$, we have an equality of $\bbZ/2$-graded $k$-algebras:
$$
\cP^{\mathrm{nc}}(\{(V_n,W_n,\omega_n)\}_{n \in \bbZ})= \cP^{\mathrm{nc}}(\oplus_{m \in \bbZ}(\{(V_n,W_n,\omega_n)\}_{n \in \bbZ} \otimes H^2_{dRB}(\bbP^1)^{\otimes (-m)}))\,.
$$
\end{lemma}
\begin{proof}
Let $(V,W,\omega)$ and $(V'W',\omega')$ be two objects of $\mathrm{Vect}(k,\bbQ)$. In order to simplify the exposition, let us denote by $P(V,W,\omega) + P(V',W',\omega')$ the set of complex numbers of the form $p+p'$, with $p \in P(V,W,\omega)$ and $p' \in P(V',W',\omega')$. Recall from \S\ref{sec:ex} that $H^2_{dRB}(\bbP^1)^{\otimes (-m)}$ is the shifted triple $(k,\bbQ, \cdot (2\pi i)^{-m})[-2m]$. Consequently, the $n^{\mathrm{th}}$ component of $\{(V_n,W_n,\omega_n)\}_{n \in \bbZ} \otimes H^2_{dRB}(\bbP^1)^{\otimes (-m)}$ is given by
\begin{equation}\label{eq:component-1}
(V_{n+2m},W_{n+2m}, \omega_{n+2m})\otimes (k, \bbQ, \cdot (2\pi i)^{-m})\,.
\end{equation}
Similarly, the $n^{\mathrm{th}}$ component of $\oplus_{m \in \bbZ}(\{(V_n,W_n,\omega_n)\}_{n \in \bbZ} \otimes H^2_{dRB}(\bbP^1)^{\otimes (-m)})$ is
\begin{equation}\label{eq:component-2}
\oplus_{m\in \bbZ} (V_{n+2m},W_{n+2m}, \omega_{n+2m})\otimes (k, \bbQ, \cdot (2\pi i)^{-m})\,.
\end{equation}
Note that since $\{(V_n,W_n,\omega_n)\}_{n \in \bbZ}$ belongs to $\mathrm{Gr}^b_\bbZ(\mathrm{Vect}(k,\bbQ))$, the direct sum \eqref{eq:component-2} is finite. On the one hand, we have $P(\eqref{eq:component-1})=P(V_{n+2m},W_{n+2m}, \omega_{n+2m})(2\pi i)^{-m}$. On the other hand, $P(\eqref{eq:component-2})=+_{m \in \bbZ} P(V_{n+2m},W_{n+2m}, \omega_{n+2m})(2\pi i)^{-m}$; see \cite[Prop.~9.2.4]{Huber}. Consequently, the following equalities
\begin{eqnarray*}
&& \cup_{n \in \bbZ} P(\oplus_{m \in \bbZ} (V_{2(n+m)}, W_{2(n+m)}, \omega_{2(n+m)})\otimes (k,\bbQ, \cdot (2\pi i)^{-m}))(2 \pi i)^{-n}\\
& = &  \cup_{n \in \bbZ} (+_{m \in \bbZ} P(V_{2(n+m)}, W_{2(n+m)}, \omega_{2(n+m)})(2 \pi i)^{-m})(2\pi i)^{-n}\\ 
& = &  \cup_{n \in \bbZ} (+_{m \in \bbZ} P(V_{2(n+m)}, W_{2(n+m)}, \omega_{2(n+m)})(2 \pi i)^{-(n+m)})
\end{eqnarray*}
allow us to conclude that every degree $0$, resp. degree $1$, generator of the $\bbZ/2$-graded $k$-algebra $ \cP^{\mathrm{nc}}(\oplus_{m \in \bbZ}(\{(V_n,W_n,\omega_n)\}_{n \in \bbZ} \otimes H^2_{dRB}(\bbP^1)^{\otimes (-m)}))$ is a $k$-linear combination of degree $0$, resp. degree $1$, generators of the $\bbZ/2$-graded $k$-algebra $\cP^{\mathrm{nc}}(\{(V_n,W_n,\omega_n)\}_{n \in \bbZ})$. This implies that these two algebras are the same.
\end{proof}
\begin{proposition}\label{prop:key-last-1}
For every object $\{(V_n,W_n,\omega_n)\}_{n \in \bbZ}$ of $\mathrm{Gr}^b_\bbZ(\mathrm{Vect}(k,\bbQ))$, we have an equality of $\bbZ/2$-graded $k$-algebras:
\begin{equation}\label{eq:equality-2}
\cP^{\mathrm{nc}}(\{(V_n,W_n,\omega_n)\}_{n \in \bbZ})=\phi(\cP(\{(V_n,W_n,\omega_n)\}_{n \in \bbZ}))\,.
\end{equation}
\end{proposition}
\begin{proof}
Recall from \S\ref{sec:periods} the definition of $\bbC_{\bbZ/2}^{2\pi i}$. Note that the underlying $\bbZ/2$-graded $\bbC$-vector space of $\bbC^{2\pi i}_{\bbZ/2}$ is $\bbC_0\oplus \bbC_1$ and that the multiplication law is given by
\begin{equation}\label{eq:law}
(\lambda_0,\lambda_1) \cdot (\lambda'_0,\lambda'_1)=(\lambda_0\lambda'_0+ \lambda_1 \lambda'_1 (2\pi i)^{-1}, \lambda_0 \lambda'_1 + \lambda_1 \lambda'_0)\,.
\end{equation}
Recall also from \S\ref{sec:periods} that the $\bbZ$-graded $k$-algebra $\cP(\{(V_n,W_n, \omega_n)\}_{n \in \bbZ})$ is generated in degree $2n$ by the elements of the set $P(V_{2n}, W_{2n}, \omega_{2n})$ and in degree $2n+1$ by the elements of the set $P(V_{2n+1}, W_{2n+1}, \omega_{2n+1})$. Via the above description \eqref{eq:law}, the image of $P(V_{2n}, W_{2n}, \omega_{2n})$ under the quotient homomorphism $\phi\colon \bbC[t,t^{-1}] \twoheadrightarrow \bbC^{2 \pi i}_{\bbZ/2}$ corresponds to the subset $P(V_{2n}, W_{2n}, \omega_{2n})(2\pi i)^{-n}\subseteq \bbC_0$. Similarly, the image of $P(V_{2n+1}, W_{2n+1}, \omega_{2n+1})$ under $\phi$ corresponds to the subset $P(V_{2n+1}, W_{2n+1}, \omega_{2n+1})(2\pi i)^{-n}\subseteq \bbC_1$. By definition of $\cP^{\mathrm{nc}}(\{(V_n, W_n, \omega_n)\}_{n \in \bbZ})$, we then obtain the searched equality \eqref{eq:equality-2}.
\end{proof}
Given two objects $(V,W,\omega)$ and $(V',W',\omega')$ of the category $\mathrm{Ind}(\mathrm{Vect}(k,\bbQ))$, the subset $P((V,W,\omega)\oplus(V',W',\omega'))\subseteq \bbC$ consists of the complex numbers of the form $p+p'$, with $p \in P(V,W,\omega)$ and $p' \in P(V',W',\omega')$. Therefore, given objects $\{(V_n,W_n,\omega_n)\}_{n \in \bbZ}$ and $\{(V'_n,W'_n,\omega'_n)\}_{n \in \bbZ}$ of $\mathrm{Gr}_\bbZ(\mathrm{Ind}(\mathrm{Vect}(k,\bbQ)))$, we observe that the $\bbZ/2$-graded $k$-algebra $\cP^{\mathrm{nc}}(\{(V_n,W_n,\omega_n)\}_{n \in \bbZ}\oplus \{(V'_n,W'_n,\omega'_n)\}_{n \in \bbZ})$ agrees with the smallest $\bbZ/2$-graded $k$-subalgebra
$$ \cP^{\mathrm{nc}}(\{(V_n,W_n,\omega_n)\}_{n \in \bbZ}) \diamond \cP^{\mathrm{nc}}(\{(V'_n,W'_n,\omega'_n)\}_{n \in \bbZ}) \subseteq \bbC^{2\pi i}_{\bbZ/2}$$
containing $\cP^{\mathrm{nc}}(\{(V_n,W_n,\omega_n)\}_{n \in \bbZ})$ and $\cP^{\mathrm{nc}}(\{(V'_n,W'_n,\omega'_n)\}_{n \in \bbZ})$.

Let $\cA, \cB \subseteq \cC$ be dg categories yielding a semi-orthogonal decomposition $\dgHo(\cC)=\langle \dgHo(\cA), \dgHo(\cB)\rangle$. Proposition \ref{prop:new} implies that $(H^\ast_{dRB})^{\mathrm{nc}}(\cC)$ is isomorphism to the direct sum $(H^\ast_{dRB})^{\mathrm{nc}}(\cA) \oplus (H^\ast_{dRB})^{\mathrm{nc}}(\cB)$. By applying $\cP^{\mathrm{nc}}(-)$, the above considerations allow us then to conclude that $\cP^{\mathrm{nc}}(\cC)=\cP^{\mathrm{nc}}(\cA) \diamond \cP^{\mathrm{nc}}(\cB)$. This proves item (iii) and finishes the proof of Theorem \ref{thm:main4}.
\begin{remark}[Periods of noncommutative mixed motives]
Recall from \S\ref{sec:proof} that the new additive invariant $H^{\mathrm{nc}}_{dRB}$ factors through the functor $U(-)_R$. Therefore, similarly to Definition \ref{def:periods}, we can define the $\bbZ/2$-graded algebra of periods $\cP^{\mathrm{nc}}(N)$ of every noncommutative mixed motive $N \in \Mot(k;R)$. Here is an example: let $k$ be a number field. Recall from \cite[\S20.3]{Andre} that the $1$-motives $K(q):=[\bbZ \stackrel{1\mapsto q}{\too} \bbG_m]$, with $q \in k^\times$, are called the {\em Kummer motives}. Following \eqref{eq:K-spectra}, the morphisms in the triangulated category $\Mot(k;\bbQ)$ from $U(k)_\bbQ$ and $U(k)_\bbQ[-1]$ are in bijection with the elements of $K_1(k)_\bbQ = k^\times \otimes \bbQ$. Therefore, we can consider the triangle:
$$ U(k)_\bbQ[-2] \too K^{\mathrm{nc}}(q) \too U(k)_\bbQ \stackrel{q}{\too} U(k)_\bbQ[-1]\,.$$
Since $K(q)$, considered as an object of $\mathrm{DM}_{\mathrm{gm}}(k;\bbQ)$, is an extension of $\bbQ(0)$ by $\bbQ(1)$, we have an isomorphism between $(H^\ast_{dRB})^{\mathrm{nc}}(K^{\mathrm{nc}}(q))$ and $\mathrm{H}_{dRB}^\ast(K(q))$. Making use of (the generalization of) Theorem \ref{thm:main4}(ii), we then conclude that $\cP^{\mathrm{nc}}(K^{\mathrm{nc}}(q))=\phi(\cP(K(q)))$. As explained in \cite[\S23.3.3]{Andre}, the transcendental number $\mathrm{log}(q)$ belongs to $\cP(K(q))_1$. Consequently, it belongs also to $\cP^{\mathrm{nc}}(K^{\mathrm{nc}}(q))_1$.
\end{remark}
\section{Proof of Theorem \ref{thm:HPD}}
By definition of $\perf(X)=\langle \bbA_0, \bbA_1(1), \ldots, \bbA_{i-1}(i-1)\rangle$, we have a chain of admissible triangulated subcategories $\bbA_{i-1}\subseteq \cdots \subseteq \bbA_1\subseteq \bbA_0$ and $\bbA_r(r):=\bbA_r \otimes \cO_X(r)$. Let $\mathfrak{a}_r$ be the right orthogonal complement to $\bbA_{r+1}$ in $\bbA_r$; these are called the {\em primitive subcategories} in \cite[\S4]{KuznetsovHPD}. Note that we have semi-orthogonal decompositions:
\begin{eqnarray}\label{eq:s-1}
\bbA_r = \langle \mathfrak{a}_r, \mathfrak{a}_{r+1}, \ldots, \mathfrak{a}_{i-1} \rangle && 0\leq r \leq i-1\,.
\end{eqnarray}
As proved in \cite[Thm.~6.3]{KuznetsovHPD}, the category $\perf(Y)$ admits a HP-dual Lefschetz decomposition $\perf(Y)=\langle \bbB_{j-1}(1-j), \bbB_{j-2}(2-j), \ldots, \bbB_0\rangle$ with respect to $\cO_Y(1)$; as above we have a chain $\bbB_{j-1} \subseteq \bbB_{j-2} \subseteq \cdots \subseteq \bbB_0$ of admissible triangulated subcategories. Moreover, the primitive subcategories $\mathfrak{b}_r$ coincide with $\mathfrak{a}_r$; in this case we have semi-orthogonal decompositions:
\begin{eqnarray}\label{eq:s-2}
\bbB_r = \langle \mathfrak{a}_0, \mathfrak{a}_1, \ldots, \mathfrak{a}_{\mathrm{dim}(V)-r-2} \rangle && 0\leq r \leq j-1\,.
\end{eqnarray}
Furthermore, there exists a triangulated $\bbC_L$ and semi-orthogonal decompositions
\begin{equation}\label{eq:semi-1}
\perf(X_L)=\langle \bbC_L, \bbA_{\mathrm{dim}(L)}(1), \ldots, \bbA_n(i-\mathrm{dim}(L))\rangle 
\end{equation}
\begin{equation}\label{eq:semi-2}
\perf(Y_L)=\langle \bbB_{j-1}(\mathrm{dim}(L^\perp)-j), \ldots, \bbB_{\mathrm{dim}(L^\perp)}(-1), \bbC_L \rangle\,.
\end{equation}
Let us denote by $\bbC_L^\dg, \bbA_r^\dg, \mathfrak{a}_r^\dg$ the dg enhancement of $\bbC_L, \bbA_r, \mathfrak{a}_r$ induced from $\perf_\dg(X_L)$. Similarly, let $\bbC_L^{\dg'}, \bbB_r^{\dg}, \mathfrak{a}_r^{\dg'}$ be the dg enhancement of $\bbC_L, \bbB_r, \mathfrak{a}_r$ induced from $\perf_\dg(Y_L)$. Since the functor $\perf(X_L) \to \bbC_L \to \perf(Y_L)$, as well as  the identification between $\mathfrak{a}_r$ and $\mathfrak{b}_r$, is of Fourier-Mukai type, we have derived Morita equivalences $\bbC_L^\dg \simeq \bbC_L^{\dg'}$ and $\mathfrak{a}_r^\dg \simeq \mathfrak{a}_r^{\dg'}$. By combining the semi-orthogonal decompositions \eqref{eq:s-1} and \eqref{eq:semi-1}, resp. \eqref{eq:s-2} and \eqref{eq:semi-2} and the equality $\mathrm{dim}(V)=\mathrm{dim}(L) + \mathrm{dim}(L^\perp)$, with Theorem \ref{thm:main4} we hence conclude that 
\begin{equation}\label{eq:conclusion-1}
\phi(\cP(X_L))= \cP^{\mathrm{nc}}(\bbC_L^\dg)\diamond \cP^{\mathrm{nc}}(\mathfrak{a}^\dg_{\mathrm{dim}(L)}) \diamond \cdots \diamond \cP^{\mathrm{nc}}(\mathfrak{a}_{i-1}^\dg)
\end{equation}
\begin{equation}\label{eq:conclusion-2}
\phi(\cP(Y_L))=\cP^{\mathrm{nc}}(\mathfrak{a}_0^\dg)\diamond \cdots \diamond \cP^{\mathrm{nc}}(\mathfrak{a}_{\mathrm{dim}(L)-2})\diamond \cP^{\mathrm{nc}}(\bbC_L^\dg)\,.
\end{equation}
On the one hand, the assumption that $\bbA_0$ is generated by exceptional objects implies that $\cP^{\mathrm{nc}}(\bbA_0^\dg)=k$. On the other hand, the semi-orthogonal decomposition \eqref{eq:s-1} (with $r=0$) implies that $\cP^{\mathrm{nc}}(\bbA_0^\dg)= \cP^{\mathrm{nc}}(\mathfrak{a}_0^\dg)\diamond \cdots \diamond \cP^{\mathrm{nc}}(\mathfrak{a}_{i-1}^\dg)$. This allows us to conclude that $\cP^{\mathrm{nc}}(\mathfrak{a}_r^\dg)=k$ for every $0 \leq r \leq i-1$. Therefore, from \eqref{eq:conclusion-1}-\eqref{eq:conclusion-2} we obtain the searched equality $\phi(\cP(X_L))=\phi(\cP(Y_L))$. This finishes the proof.

\medbreak\noindent\textbf{Acknowledgments:} The author is grateful to Joseph Ayoub for an e-mail exchange about realizations, to Guillermo Corti\~nas for a conversation about $K$-theory, to Christian Haesemeyer for explanations about $cdh$-descent, to Maxim Kontsevich for discussions about noncommutative motives and noncommutative periods, to Marc Levine for a conversation about realizations, and finally to Michel Van den Bergh for several discussions about noncommutative motives and $\cD$-modules.


\begin{thebibliography}{00}

\bibitem{Andre} Y.~Andr\'e, {\em Une introduction aux motifs (motifs purs, motifs mixtes, p\'eriodes)}. Panoramas et Synth\`eses {\bf 17}. Soci\'et\'ee Math\'ematique de France, Paris, 2004.
 
 
 
 \bibitem{Hopf2} J.~Ayoub, {\em L'alg\`ebre de Hopf et le groupe de Galois motiviques d'un corps de caracteristique nulle, II}. Journal f\"ur die reine und angewandte Mathematik {\bf 693} (2014), 151--226.
 
\bibitem{Ayoub} \bysame, {\em Les six op{\'e}rations de Grothendieck et le formalisme des cycles {\'e}vanescents dans le monde motivique}. Ast{\'e}risque {\bf 314}--{\bf 315} (2007), Soci{\'e}t{\'e} Math{\'e}matique de France.

\bibitem{BeiBer} A.~Beilinson and J.~Bernstein, {\em Localisation de $\mathfrak{g}$-modules}. C. R. Acad. Sci. Paris S\'er. I Math. {\bf 292} (1981), no.~1, 15--18. 

\bibitem{BMR} R.~Bezrukavnikov, I.~Mirkovi\'c and D.~Rumynin, {\em Localization of modules for a semisimple Lie algebra in prime characteristic}. Ann. of Math. (2) {\bf 167} (2008), no. 3, 945--991.

\bibitem{BB} A.~Bialynicki-Birula, {\em Some theorems on actions of algebraic groups}. Ann. of Math. (2), {\bf 98}
(1973), 480--497.

\bibitem{Bloch} S.~Bloch, {\em Algebraic cycles and higher $K$-theory}. Adv. Math. {\bf 61} (3) (1986) 267--304.

\bibitem{BO} A.~Bondal and D.~Orlov, {\em Semiorthogonal decomposition for algebraic varieties}. Available at arXiv:alg-geom/9506012.

\bibitem{BV} A.~Bondal and M.~Van den Bergh, {\em Generators and representability of functors in commutative and noncommutative geometry}. Mosc. Math. J. {\bf 3} (2003), no. 1, 1--36, 258.

\bibitem{DM} P.~Deligne and J.~S.~Milne, {\em Tannakian categories. Hodge cycles, motives, and Shimura varieties}. LNM {\bf 900} (1982), 101--228. Available at \url{http://www.jmilne.org/math/xnotes/tc.pdf.}

\bibitem{CG} U.~Choudhury and M.~Gallauer Alves de Souza, {\em An isomorphism of motivic Galois groups}. Available at arXiv:1410.6104.




\bibitem{Drinfeld} V.~Drinfeld, {\em DG quotients of DG categories}. J. Algebra {\bf 272} (2004), 643--691.

\bibitem{GD} V.~Drinfeld and D.~Gaitsgory, {\em Compact generation of the category of $D$-modules on the stack of $G$-bundles on a curve}. Camb. J. Math. {\bf 3} (2015), no. 1-2, 19--125. 

\bibitem{Haesemeyer} C.~Haesemeyer, {\em Descent properties of homotopy $K$-theory}. Duke Math. J. {\bf 125} (2004), no.~3, 589--620. 

\bibitem{Hironaka} H. Hironaka, {\em Resolution of singularities of an algebraic variety over a field of characteristic zero, I, II}. Ann. of Math. (2) {\bf 79} (1964), 109--203; 205--326.

\bibitem{Hodge} T.~Hodges, {\em $K$-theory of Noetherian rings}. S\'eminaire d'Alg\`ebre Paul Dubreil et Marie-Paul Malliavin, 39\`eme Ann\'ee (Paris, 1987/1988), 246--268, 
LNM {\bf 1404}, Springer, Berlin, 1989. 

\bibitem{Huber} A. Huber and S.~M\"uller-Stach, {\em Periods and Nori motives}. Book available at \url{http://home.mathematik.uni-freiburg.de/arithgeom/forschung.html}.

\bibitem{Jannsen} U. Jannsen, {\em Mixed motives and algebraic $K$-theory}. 
With appendices by S. Bloch and C. Schoen. LNM {\bf 1400}. Springer-Verlag, Berlin, 1990.

\bibitem{Kapranov} M.~Kapranov, {\em On the derived categories of coherent sheaves on some homogeneous spaces}. Invent. Math. {\bf 92} (1988), no. 3, 479--508.

\bibitem{ICM-Keller} B.~Keller, {\em On differential graded categories}. International Congress of Mathematicians (Madrid), Vol.~II,  151--190. Eur.~Math.~Soc., Z{\"u}rich (2006).




\bibitem{Miami} M.~Kontsevich, {\em Mixed noncommutative motives}. Talk at the Workshop on Homological Mirror Symmetry,  Miami, 2010. Notes available at \url{www-math.mit.edu/auroux/frg/miami10-notes}.  

\bibitem{finMot} \bysame, {\em Notes on motives in finite characteristic}.  Algebra, arithmetic, and geometry: in honor of Yu. I. Manin. Vol. II,  213--247, Progr. Math., {\bf 270}, BirkhŠuser Boston, MA, 2009. 

\bibitem{Kontsevich-talk} \bysame, {\em Categorification, NC Motives, Geometric Langlands, and Lattice Models}. Talk at the Geometric Langlands Seminar, University of Chicago, 2006. Notes available at \url{https://www.ma.utexas.edu/users/benzvi/notes.html}

\bibitem{IAS} \bysame, {\em Noncommutative motives}. Talk at the IAS on the occasion of the $61^{\mathrm{st}}$ birthday of Pierre Deligne (2005). Available at \url{http://video.ias.edu/Geometry-and-Arithmetic}.    


\bibitem{KZ} M.~Kontsevich and D.~Zagier, {\em Periods}. Mathematics unlimited--2001 and beyond, 771--808, Springer, Berlin, 2001.

\bibitem{KuznetsovFib} A.~Kuznetsov, {\em Derived categories of quadric fibrations and intersections of quadrics}. Adv. Math. {\bf 218} (2008), no. 5, 1340--1369.

\bibitem{KuznetsovHPD} \bysame, {\em Homological projective duality}. Publ. Math. IH\'ES (2007), no. {\bf 105}, 157--220.

\bibitem{ICM-Kuznetsov} \bysame, {\em Semiorthogonal decompositions in algebraic geometry}. Available at 1404.3143. To appear in Proceedings of the ICM 2014.

\bibitem{Kuznetsov1} \bysame, {\em Homological projective duality for Grassmannians of lines}. Available at arXiv:0610957.

    



\bibitem{Loday} J.-L.~Loday, {\em Cyclic homology}. Grundlehren der Mathematischen Wissenschaften, vol. {\bf 301}, Springer-Verlag, Berlin, 1998.



\bibitem{JEMS} M.~Marcolli and G.~Tabuada, {\em Noncommutative numerical motives, Tannakian structures, and motivic Galois groups}. J. Eur. Math. Soc. (JEMS) {\bf 18} (2016), no. 3, 623--655. 

\bibitem{MV} F.~Morel and V.~Voevodsky, {\em $\bbA^1$-homotopy theory of schemes}. Inst. Hautes {\'E}tudes Sci. Publ. Math. (1999), no. {\bf 90}, 45--143 (2001).

\bibitem{Neeman} A.~Neeman, {\em Triangulated categories}. Annals of Mathematics Studies {\bf 148}, Princeton University Press, Princeton, NJ, 2011.

\bibitem{Quillen} D.~Quillen, {\em Higher algebraic K-theory. I}. Algebraic K-theory, I: Higher K-theories (Proc. Conf., Battelle Memorial Inst., Seattle, 1972), Springer, Berlin, 1973, pp. 85--147. LNM {\bf 341}.

\bibitem{Riou} J.~Riou, {\em Algebraic K-theory, $\bbA^1$-homotopy and Riemann-Roch theorems}. J. Topol. {\bf 3} (2010), no.~2, 229--264. 

\bibitem{RO1} O.~R{\"o}ndigs and P.~A.~{\O}stv{\ae}r, {\em Modules over motivic cohomology}. Adv. Math. {\bf 219} (2008), no. 2, 689--727. 

\bibitem{RSO} O.~R\"ondigs, M.~Spitzweck and P.~{\O}stv{\ae}r, {\em Motivic strict ring models for $K$-theory}. Proc. Amer. Math. Soc. {\bf 138} (2010), no.~10, 3509--3520.

\bibitem{Steenbrink} J. Steenbrink, {\em A summary of mixed Hodge theory}. Motives I. Proc. Symp.
Pure Math. Vol. {\bf 55}, AMS 1991, pp. 31--42.



\bibitem{Hopf} G.~Tabuada, {\em Noncommutative mixed (Artin) motives and their motivic Hopf dg algebras}. Selecta Mathematica, {\bf 22} (2016), no.~2, 735--764.

\bibitem{book} \bysame, {\em Noncommutative Motives}. With a preface by Yuri I. Manin. University Lecture Series {\bf 63}. American Mathematical Society, Providence, RI, 2015.

\bibitem{A1-homotopy} \bysame, {\em $\bbA^1$-theory of noncommutative motives}. JNCG {\bf 9}  (2015), no. 3, 851--875. 

\bibitem{Bridge} \bysame, {\em Voevodsky's mixed motives versus Kontsevich's noncommutative mixed motives}. Adv. Math. {\bf 264} (2014), 506--545.


\bibitem{Universal} \bysame, {\em Universal suspension via noncommutative motives}. JNCG {\bf 5} (2011), no. 4, 573--589.

\bibitem{Additive} \bysame, {\em Invariants additifs de DG-cat{\'e}gories}. Int. Math. Res. Not. (2005), no. {\bf 53}, 3309--3339.

\bibitem{Azumaya} G.~Tabuada and M.~Van den Bergh, {\em Noncommutative motives of Azumaya algebras}. J. Inst. Math. Jussieu {\bf 14} (2015), no. {\bf 2}, 379--403.

\bibitem{Gysin} \bysame, {\em The Gysin triangle via localization and $\bbA^1$-homotopy invariance}. Available at arXiv:1510.04677. To appear in Transactions of the AMS.


\bibitem{TT} R.~W.~Thomason and T.~Trobaugh, {\em Higher algebraic $K$-theory of schemes and of derived  categories}. Grothendieck Festschrift, Volume III. Volume {\bf 88} of Progress in Math., 247--436. Birkhauser, Boston, Bassel, Berlin, 1990. 

\bibitem{JPAA} V.~Voevodsky, {\em Unstable motivic homotopy categories in Nisnevich and $cdh$-topologies}.
J. Pure Appl. Algebra {\bf 214} (2010), no.~8, 1399--1406. 

\bibitem{Voevodsky} \bysame, {\em Triangulated categories of motives over a field}. Cycles, transfers, and motivic homology theories, Ann. of Math. Stud., vol. {\bf 143}, Princeton Univ. Press, 2000, pp. 188--238.

\bibitem{Voevodsky-ICM} \bysame, {\em $\bbA^1$-homotopy theory}. Proceedings of the International Congress of Mathematicians, Vol. I (Berlin, 1998), no. Extra Vol. I, 1998, pp. 579--604 (electronic).


\bibitem{Wodzicki} M.~Wodzicki, {\em Cyclic homology of differential operators}. Duke Math. J. {\bf 54} (1987), no. 2, 641--647. 

\end{thebibliography}
\end{document}